\documentclass[10pt]{amsart}
\usepackage[english]{babel}
\usepackage{amsthm,amsmath,amsfonts}
\usepackage[all]{xy}
\usepackage{color}
\usepackage[utf8]{inputenc} 
\usepackage[ colorlinks=true, linkcolor=blue, citecolor=blue]{hyperref}
\usepackage{tikz-cd}
\DeclareFontFamily{U}{mathc}{}
\DeclareFontShape{U}{mathc}{m}{it}%
{<->s*[1.03] mathc10}{}

\DeclareMathAlphabet{\mathscr}{U}{mathc}{m}{it}

\numberwithin{equation}{section} 

\allowdisplaybreaks

\newtheorem{theorem}{Theorem}[section]
\newtheorem{lemma}[theorem]{Lemma}

\newtheorem{proposition}[theorem]{Proposition}

\newtheorem{corollary}[theorem]{Corollary}

\newcommand{\At}{\operatorname{At}}
\newcommand{\Ad}{\operatorname{Ad}}
\newcommand{\Ext}{\operatorname{Ext}}
\theoremstyle{definition}
\newtheorem{definition}[theorem]{Definition}
\newtheorem{example}[theorem]{Example}

\newtheorem*{ackno}{Acknowledgements}

\theoremstyle{remark}
\newtheorem{remark}[theorem]{Remark}
\newcommand{\sE}{\mathcal{E}}
\newcommand{\sJ}{\mathcal{J}}
\newcommand{\sF}{\mathcal{F}}
\newcommand{\sQ}{\mathcal{Q}}
\newcommand{\Oh}{\mathcal{O}}
\newcommand{\sU}{\mathcal{U}}
\newcommand{\sD}{\mathcal{D}}
\newcommand{\K}{\mathbb{K}}
\newcommand{\N}{\mathbb{N}}
\newcommand{\Gau}{\operatorname{Gauge}}
\newcommand{\Sad}{\mathscr{ad}(P)}
\newcommand{\LL}{\mathcal{L}}
\renewcommand{\bar}[1]{\overline{#1}}

\newcommand{\Art}{\mathbf{Art}}
\newcommand{\Set}{\mathbf{Set}}

\newcommand{\dr}{d_{dR}}

\newcommand{\Def}{\operatorname{Def}}

\newcommand{\Tot}{\operatorname{Tot}}

\newcommand{\Ker}{\operatorname{Ker}}
\newcommand{\Id}{\operatorname{Id}}
\newcommand{\ad}{\operatorname{ad}}
\newcommand{\Tr}{\operatorname{Tr}}
\newcommand{\g}{\mathfrak{g}}
\newcommand{\la}{\langle}
\newcommand{\ra}{\rangle}
\newcommand{\na}{\nabla}
\newcommand{\Spec}{\operatorname{Spec}}

\newcommand{\DER}{{{\mathcal D}er}}
\newcommand{\HOM}{{{\mathcal H}om}}

\newcommand{\m}{\mathfrak{m}}

\title{Cyclic forms on DG-Lie algebroids and semiregularity}

\author{Emma Lepri}
\address{\newline
	Universit\`a degli studi di Roma La Sapienza,\hfill\newline
	Dipartimento di Matematica  Guido
	Castelnuovo,\hfill\newline
	P.le Aldo Moro 5,
	I-00185 Roma, Italy.\medskip}
\email{emma.lepri@uniroma1.it}

\begin{document}
	
	\maketitle
	
	\begin{abstract} \begin{sloppypar} Given a transitive DG-Lie algebroid $(\mathcal{A}, \rho)$ over a smooth separated scheme $X$ of finite type over a field $\mathbb{K}$ of characteristic 0 we define a notion of connection ${\nabla \colon \mathbf{R}\Gamma(X,\mathrm{Ker} \rho) \to \mathbf{R}\Gamma (X,\Omega_X^1[-1]\otimes \mathrm{Ker} \rho)}$ 
	 and construct an $L_\infty$ morphism between DG-Lie algebras $f \colon \mathbf{R}\Gamma(X, \mathrm{Ker} \rho) \rightsquigarrow\mathbf{R}\Gamma(X, \Omega_X^{\leq 1} [2])$
		associated to a connection and to a cyclic form on the DG-Lie algebroid.
				In this way, we obtain a lifting of the first component of the modified Buchweitz-Flenner semiregularity map in the algebraic context, which has an application to the deformation theory of coherent sheaves on $X$ admitting a finite locally free resolution.
				Another application is to the deformations of (Zariski) principal bundles on $X$.
				\end{sloppypar}
	\end{abstract}

\section*{Introduction}

Let $\sF$ be a coherent sheaf admitting a finite locally free resolution on a smooth variety $X$ over a field $\K$ of characteristic zero.
The Buchweitz-Flenner semiregularity map, introduced in \cite{BF}, and generalising the  semiregularity map of Bloch \cite{Blo}, is  
defined by the formula:
\begin{equation}\label{eq.BF}
\sigma\colon \Ext^2_X(\sF,\sF)\to \prod_{q\ge 0}H^{q+2}(X,\Omega_X^q),\qquad 
\sigma(c)=\Tr(\exp(-\At(\sF))\circ c), 
\end{equation} where $\Tr$ denotes the trace maps
 $\Tr\colon \Ext_X^i(\sF,\sF\otimes \Omega^j_X)\to H^i(X,\Omega^j_X)$ for  $i,j\geq0$, and the exponential of the opposite of the Atiyah class $\At(\sF) \in \Ext_X^1(\sF, \sF \otimes \Omega_X^1)$ is defined via the  Yoneda pairing \[ \Ext^i_X(\sF,\sF\otimes \Omega^i_X)\times \Ext^j_X(\sF,\sF\otimes \Omega^j_X)\to 
\Ext^{i+j}_X(\sF,\sF\otimes \Omega^{i+j}_X),\qquad (a,b)\mapsto a\circ b,\] 
\[ \exp(-\At(\sF))\in \prod_{q\ge 0}\Ext^q_X(\sF,\sF\otimes \Omega^q_X). \]

We refer to \cite{linfsemireg} for a discussion of the role of the Buchweitz-Flenner semiregularity map in deformation theory and of the reason it is more convenient in this setting to consider the modified Buchweitz-Flenner semiregularity map, obtained as follows. 
Denote by $\sigma_q \colon \Ext_X^2(\sF, \sF) \to H^{q+2}(X, \Omega_X^q)$  the components of the semiregularity map $\sigma = \sum \sigma_q$,
for every $q \geq 0$ denote by $\Omega^{\le q}_X=(\oplus_{i=0}^q\Omega_X^i[-i],\dr)$ the truncated de Rham complex, and  consider the composition 
\[  \tau_q\colon \Ext_X^2(\sF,\sF)\xrightarrow{\sigma_q}H^{q+2}(X,\Omega_X^q)=H^{2}(X,\Omega_X^q[q])\xrightarrow{i_q} \mathbb{H}^{2}(X,\Omega_X^{\le q}[2q]),\]
where the map $i_q$ is induced by the inclusion of complexes $\Omega_X^{q}[q]\subset \Omega_X^{\le q}[2q]$. The map $\tau_q$ is the $q$-component of the modified Buchweitz-Flenner semiregularity map.

A lifting to an $L_\infty$ morphism of the first component $\tau_1$ 
 of the modified Buchweitz-Flenner semiregularity map was constructed in \cite{linfsemireg} in the context of complex manifolds: for every 
connection of type $(1,0)$ on a
finite complex of locally free sheaves $\sE$ on a complex manifold $X$ the existence of an $L_{\infty}$ morphism between DG-Lie algebras \[ g\colon A^{0, *}_X(\HOM^*_{\Oh_X}(\sE,\sE))\rightsquigarrow \dfrac{A^{*,*}_X}{A^{\ge 2,*}_X}[2]\] 
whose linear component induces $\tau_1$ in cohomology was proved. 
As a consequence, recalling that $\Ext_X^2(\sF, \sF)$ is the obstruction space for the functor of deformations of a coherent sheaf $\sF$, the map $\tau_1$ annihilates all obstructions to deformations of a coherent sheaf admitting a finite locally free resolution. We refer again to \cite{linfsemireg} for a survey on the existing literature in this regard; we only note here that this fact was proved by Mukai and Artamkin \cite{Arta} for the $0$th component of the Buchweitz-Flenner semiregularity map, in \cite{semireg} with some mild assumptions for the Bloch semiregularity map, in \cite{BF} for curvilinear obstructions for the map $\sigma$ when the Hodge to de Rham spectral sequence of $X$ degenerates at $E_1$, and finally in the general case for all obstructions and for the maps $\tau_q$ in \cite{Pri}.

The fact that the construction of the $L_\infty$ morphism  can also be realised  in the algebraic case was outlined in \cite[Section 5]{linfsemireg} and is expanded on here: given a simplicial connection on the finite complex of locally free sheaves $\sE$, the map $\tau_1$ can be lifted to an $L_\infty$ morphism 
	\[g\colon \Tot(\mathcal{U},\HOM^*_{\Oh_X} (\sE,\sE)) \rightsquigarrow \Tot(\mathcal{U}, \Omega_X^{\leq 1} [2]),\]
	where $\Tot$ denotes the Thom-Whitney totalisation and $\sU$ is an affine open cover of $X$.
	
Since $\HOM_{\Oh_X}^*(\sE,\sE)$ is the kernel of the anchor map of the transitive DG-Lie algebroid of derivations of pairs $\sD^*(X,\sE)$ of \cite{DMcoppie}, it is natural to generalise this construction to the framework of DG-Lie algebroids. In fact, the main result of this paper, Theorem~\ref{teo.linfinitoalg}, is the construction  of an $L_\infty$ morphism between DG-Lie algebras 
\[ f \colon \Tot(\sU, \LL) \rightsquigarrow \Tot(\sU, \Omega_X^{\leq 1}[2]),\]
where $\LL$ denotes the kernel of the anchor map of a transitive DG-Lie algebroid,  which is equipped 
with a connection $\na \colon \Tot(\sU, \LL) \to \Tot(\sU, \Omega^1_X[-1]\otimes \LL)$ and a $d_{\Tot}$-closed cyclic form. 

For this reason, in Section~\ref{sec.DGLiealg} we introduce connections on transitive DG-Lie algebroids, which are $\K$-linear operators that play the role of connections of type $(1,0)$ in the construction of the $L_\infty$ morphism. Connections on transitive DG-Lie algebroids are associated to
simplicial liftings of the identity; also associated to a simplicial lifting of the identity is the extension cocycle, which generalises the notion of Atiyah cocycle.
Section~\ref{sec.cyclic} describes cyclic forms on DG-Lie algebroids, cyclic forms induced by DG-Lie algebroid representations and the construction of the $L_\infty$ morphism.
Also contained in Section~\ref{sec.cyclic} is the following application to the deformation theory of coherent sheaves, analogous to the one  obtained for complex manifolds in \cite{linfsemireg}:

\begin{theorem}[=Corollary~\ref{cor.fasci2}]
		Let $\sF$ be a coherent sheaf admitting a finite locally free resolution on a smooth separated scheme $X$ of finite type over a field $\K$ of characteristic zero. Then every obstruction to the deformations of $\sF$ belongs to the kernel of the map
	\[  \tau_1\colon \Ext_X^2(\sF,\sF)\to \mathbb{H}^{2}(X,\Omega_X^{\le 1}[2]).\]
	{If the Hodge to de Rham spectral sequence of $X$ degenerates at $E_1$, then  
		every obstruction to the deformations of $\sF$ belongs to the kernel of the map
		\[  \sigma_1\colon \Ext_X^2(\sF,\sF)\to {H}^{3}(X,\Omega_X^{1}),\qquad \sigma_1(a)=-\Tr(\At(\sF)\circ a).\]}
\end{theorem}

Lastly, since to every principal bundle one can naturally associate the Atiyah Lie algebroid of \cite{At}, Section~\ref{sec.fibratiprinc} contains the following application to the deformation theory of (Zariski) principal bundles. 
	\begin{theorem}[=Corollary~\ref{cor.principali}]
		Let $P$ be a principal bundle on a smooth separated scheme $X$ of finite type over an algebraically closed field $\K$ of characteristic zero and let $$\la-,-\ra \colon \Tot(\mathcal{U}, \Omega^i_X[-i] \otimes \Sad) \times \Tot(\mathcal{U}, \Omega^j_X[-j] \otimes \Sad) \to \Tot(\mathcal{U}, \Omega^{i+j}_X[-i-j] ),\quad i,j \geq 0,$$ be a $d_{\Tot}$-closed cyclic  form.
		Then every obstruction to the deformations of $P$ belongs to the kernel of the map
		\begin{equation*}
		f_1 \colon H^2 (X, \Sad) \to \mathbb{H}^2 (X, \Omega_X^{\leq 1}[2]), \quad f_1(x)= \la \At(P), x \ra,
		\end{equation*}
		where $\At(P)$ denotes the Atiyah class of the principal bundle $P$.
\end{theorem}

\subsection*{Notation} By $\K$ we always denote a characteristic zero field. Given two complexes of $\Oh_X$-modules $\sF$ and $\mathcal{G}$, $\sF \otimes \mathcal{G}$ denotes $\sF \otimes_{\Oh_X} \mathcal{G}$.   If $V=\oplus V^i$ is either a graded vector space or a graded sheaf, $\bar{v}$ denotes the degree of a homogeneous element $v\in V$. For every integer $p$ the symbol $[p]$ denotes the shift functor, defined by $V[p]^i=V^{p+i}$. For complexes of $\Oh_X$-modules $\sE,\sF$  we denote by $\HOM^*_{\Oh_X}(\sE,\sF)$ the graded sheaf of $\Oh_X$-linear morphisms	\[ \HOM^*_{\Oh_X}(\sE,\sF)=\bigoplus_i \HOM^i_{\Oh_X}(\sE,\sF),\qquad 
\HOM^i_{\Oh_X}(\sE,\sF)=\prod_j\HOM_{\Oh_X}(\sE^j,\sF^{i+j})\,.\]

	\section{DG-Lie algebroids, connections and extension cocycles}\label{sec.DGLiealg}

	The goal of this section is to define $\K$-linear operators 
	$$ \na \colon \Tot(\sU, \LL) \to \Tot(\sU, \Omega_X^1[-1] \otimes \LL)$$ called 
	connections on the kernel $\LL$ of the anchor map of a transitive DG-Lie algebroid. A short review of the Thom-Whitney totalisation functor $\Tot$ is given. In order to construct a connection, we introduce the notion of simplicial lifting of the identity. The section ends with the definition of the extension cocycle associated to a simplicial lifting of the identity, which generalises the notion of Atiyah cocycle. Different notions of Atiyah classes for DG-Lie algebroids have been considered elsewhere in the literature, see e.g. \cite{pairs,liealgebroids,Mehta}.

	Let $X$ be a smooth separated scheme of finite type over a field $\K$ of characteristic zero, and let $\Theta_X, \Omega^1_X$ denote its tangent and cotangent sheaves respectively. Often it will be useful to consider the cotangent sheaf as a trivial complex of sheaves concentrated in degree one, so as to have an inclusion
	$ \Omega^1_X[-1] \to \Omega_X^*$, where $\Omega_X^* = \oplus_p\Omega_X^p[-p]$ denotes the de Rham complex.

	\begin{definition}\label{def.DGLiealg}
		A DG-Lie algebroid over $X$ is a complex of sheaves of $\Oh_X$-modules 
		$\mathcal{A}$ equipped with a $\K$-bilinear bracket $[-,-] \colon \mathcal{A}\times \mathcal{A} \to \mathcal{A}$, which defines a
		DG-Lie algebra structure on the spaces of sections, and with a morphism
		of complexes of $\Oh_X$-modules $\rho \colon \mathcal{A} \to \Theta_X$, called the anchor map, such that the induced
		map on the spaces of sections  is a homomorphism of DG-Lie
		algebras. Moreover  for any sections $a_1 , a_2 $ of $\mathcal{A}$ and $f$  of $\Oh_X$, the following
		Leibniz identity holds:
		\begin{equation*}
		[a_1, fa_2] = f [a_1, a_2] + \rho (a_1)(f) a_2.
		\end{equation*}
	\end{definition}

\begin{example}
	The sheaf $\Theta_X$ is a trivial example of a DG-Lie algebroid concentrated in degree zero, with anchor map given by the identity.
	A DG-Lie algebroid over $\Spec \K$ is exactly a DG-Lie algebra over the field $\K$. Every sheaf of DG-Lie algebras over $\Oh_X$ can be considered as a DG-Lie algebroid over $X$ with trivial anchor map.
\end{example}

\begin{definition}
	Let $(\mathcal{A}, \rho)$ and $(\mathcal{B}, \sigma)$ be DG-Lie algebroids over $X$. A morphism of DG-Lie algebroids $\varphi \colon \mathcal{A} \to \mathcal{B}$ is a morphism of complexes of sheaves which preserves brackets and commutes with the anchor maps:
	\begin{center}
		\begin{tikzcd}
		\mathcal{A} \arrow[rr, "\varphi"] \arrow[rd, "\rho"'] &          & \mathcal{B} \arrow[ld, "\sigma"] \\
		& \Theta_X. &                                 
		\end{tikzcd}
	\end{center}
\end{definition}
Let $(\mathcal{A}, \rho)$ be a DG-Lie algebroid over $X$ and assume  that $\LL= \Ker \rho $ is a finite complex of locally free sheaves.  Notice that on $\LL$ there is a naturally induced graded Lie bracket: for  sections $x,y$ of $\LL$
\[ [x,y ]:= [i(x), i(y)],\]
where $i \colon \LL \to \mathcal{A}$ denotes the inclusion.
This bracket is $\Oh_X$-linear, in fact for any  sections $x,y$ of $\LL$ and $f$ of $\Oh_X$ one has
\[ [x,fy]:= [i(x), i(fy)]= [i(x), f i(y)]= f[i(x), i(y)] + \rho (i(x))(f) y = f[i(x), i(y)] = f[x,y],\]
so that $\LL$ is a sheaf of DG-Lie algebras over $\Oh_X$.

 \begin{definition}{\cite[Chapter 3]{Mack}}
 	A DG-Lie algebroid $(\mathcal{A}, \rho)$ over $X$ is transitive if the anchor map $\rho \colon \mathcal{A} \to \Theta_X$ is surjective.
 \end{definition}
 Let  now $(\mathcal{A}, \rho)$ be a transitive DG-Lie algebroid over $X$, consider the short exact sequence of complexes of sheaves
\begin{equation*}
	\begin{tikzcd}
		0 \arrow[r] & \LL \arrow[r, "i"] & \mathcal{A} \arrow[r, "\rho"] & \Theta_X \arrow[r] & 0
	\end{tikzcd}
\end{equation*}

and tensor it with the shifted cotangent sheaf $\Omega^1_X[-1]$ to obtain the short exact sequence
\begin{equation}\label{eq.succOmega}
	\begin{tikzcd}
	0 \arrow[r] & \Omega_X^1[-1] \otimes \LL \arrow[r, "\Id \otimes i"] & \Omega_X^1[-1] \otimes\mathcal{A} \arrow[r, "\Id \otimes \rho"] & \Omega_X^1[-1] \otimes \Theta_X \arrow[r] & 0.
	\end{tikzcd}
\end{equation}
Because of the isomorphism $\Omega_X^1[-1] \otimes \Theta_X \cong \HOM^*_{\Oh_X}(\Omega_X^1, \Omega_X^1[-1]) \cong \HOM^*_{\Oh_X}(\Omega_X^1,\Omega_X^1)[-1]$, one can consider $\Id_{\Omega^1} \in \Gamma(X, \Omega_X^1[-1] \otimes \Theta_X)$ as an element of degree one.
\begin{definition}\label{def.connglob}
	A \textbf{lifting of the identity} is a global section $D$ in $\Gamma(X,\Omega^1_X[-1] \otimes \mathcal{A})$ such that $(\Id \otimes \rho)(D) = \Id_{\Omega^1} \in \Gamma (X, \Omega_X^1[-1] \otimes \Theta_X)$. 
\end{definition}
 Since the map $\Id \otimes \rho$ is not in general surjective on global sections, a lifting of the identity does not always exist. However a germ of a lifting of the identity, i.e., a preimage of $\Id_{\Omega^1}$ in $\Omega_X^1[-1] \otimes \mathcal{A}$,  always exists.

  	\begin{example}\label{ex.coppie}
  		 For particular DG-Lie algebroids, the notion of lifting of the identity can be related to the more familiar notion of algebraic connection.
 	Let $(\sE, \delta_{\sE})$ be a finite complex of locally free sheaves. Following \cite[Section 5]{DMcoppie},  define the complex of derivations of pairs
 	\[\sD^*(X, \sE) = \{ (h,u) \in \Theta_X \times \HOM^*_\K (\sE,\sE) \ | \ u(fe)= fu(e) + h(f)e,\ \ \forall f \in \Oh_X,\ e \in \sE\}.\]
 	The complex 
 	$\sD^*(X, \sE)$ is a finite complex of coherent sheaves and the natural map
  $$\alpha \colon \sD^*(X, \sE) \to \Theta_X, \quad (h,u)\mapsto h,$$ which is called the anchor map, is surjective, see \cite{DMcoppie}. The graded Lie bracket is defined as 
 	\[ [(h,u), (h', u')] = ([h,h'], [u, u']),\] 
 	where the (graded) Lie brackets on $\Theta_X $ and $\HOM^*_\K (\sE,\sE) $ are the (graded) commutators of the composition products.  For $f \in \Oh_X$ we then have that
 	\begin{align*}
 	&[(h,u), f(h', u')]= [(h,u), (fh', fu')]= ([h,fh'], [u, fu']) =\\
 	& ( h(f) h' + f hh' - fh' h, fuu' + h(f) u' -(-1)^{\overline{u} \overline{u'}} fu'u )= (h(f)h'+f[h,h'], h(f)u' + f[u,u']) =\\
 	& f([h,h'], [u, u']) +h(f)(h', u')  = f [(h,u), (h', u')] + \alpha ((h,u)) (f) (h',u'),
 	\end{align*}
 	hence $(\sD^*(X, \sE), \alpha)$ is a transitive DG-Lie algebroid over $X$.
 	Recall that an algebraic {connection} on the complex of locally free sheaves $\sE$ is the data for every $i$ of an algebraic {connection} on $\sE^i$, i.e., a  $\K$-linear map $D \colon \sE^i \to \Omega_X^1 \otimes \sE^i$ such that for $e \in \sE^i, f \in \Oh_X$
 	\[ D (fe)= \dr f \otimes e + f D(e),\]
 	where $\dr$ denotes the universal derivation $\dr \colon \Oh_X \to \Omega_X^1$. A global algebraic connection on $\sE$ need not exist. 
 	The kernel of the anchor map $\alpha$ is the sheaf of DG-Lie algebras $\HOM^*_{\Oh_X} (\sE,\sE)$,  
 the graded sheaf of $\Oh_X$-linear endomorphisms of $\sE$,
 	with bracket equal to the graded commutator 
 	\[[f,g]= f g -(-1)^{\overline{f}\overline{g}}gf,\]
 	and differential given by 
 	\[g\mapsto [\delta_{\sE},g]=\delta_{\sE} g-(-1)^{\bar{g}}g\delta_{\sE}\,.\] 
 	The short exact sequence in \eqref{eq.succOmega} in this case is isomorphic to
 	\[
 	0\to  \HOM^*_{\Oh_X}(\sE,\Omega^1_X[-1]\otimes \sE)\xrightarrow{g\mapsto (0,g)} \sJ^*_{\Omega^1}\xrightarrow{(\beta, g)\mapsto\beta} 
 	\DER^*_{\K}(\Oh_X,\Omega^1_X[-1])\to 0\,,
 	\]
 	where the complex $\mathcal{J}^*_{\Omega^1}$ is defined as
 		$$ \{(\beta, v)\in \DER^*_{\K}(\Oh_X,\Omega_X^1[-1]) \times \HOM^*_{\K}(\sE,\Omega_X^1[-1]\otimes \sE) \mid 
 	v(fx)=fv(x)+\beta(f)\otimes x,\; \forall x\in \sE,\, f\in \Oh_X\},$$
 	see also \cite[Section 5]{linfsemireg}.
 	In this case a lifting of the identity  is exactly a global algebraic {connection} on the complex of sheaves $\sE$: via the isomorphism   $\Omega_X^1[-1] \otimes \sD^*(X, \sE) \cong \mathcal{J}^*_{\Omega^1}$ 
  a lifting of the identity $D$ corresponds to  $\K$-linear  maps $D' \colon \sE^i \to \Omega_X^1[-1] \otimes \sE^i$ for all $i$ such that $D'(fe)= fD'(e) + \dr (f) \otimes e$  for all $f \in \Oh_X$ and $e \in \sE^i$.
  \end{example}

\medspace

Before defining connections 
on the kernel of the anchor map  of a transitive DG-Lie algebroid,
	  it is useful to give a brief review of the definition and of some of the main properties of the Thom-Whitney totalisation functor $\Tot$; for more details see e.g. \cite{ sheaves,semicos,semireg, LMDT, Nav}.
   	The Thom-Whitney totalisation is a functor  from the category of semicosimplicial DG-vector spaces to the category of DG-vector spaces.
  For every $n \geq 0$ consider
   \[ A_n = \frac{\K[t_0,  \ldots, t_n,dt_0, \ldots, dt_n]}{(1- \sum_i t_i , \sum_i dt_i)}\]
   the commutative differential graded algebra of polynomial differential forms on the affine standard $n$-simplex, and the maps
   \begin{equation*}
   \delta_k^* \colon A_n \to A_{n-1},\ \ 0\leq k \leq n\quad \quad\ \ \delta_k^*(t_i)=
   \begin{cases}
   	t_i\quad i<k\\
   	0 \quad i=k\\
   	t_{i-1} \quad i>k.
   \end{cases}
  \end{equation*}
  \begin{definition}\label{def.tot}
The Thom-Whitney totalisation of a semicosimplicial DG-vector space  $V$
\begin{center}
	\begin{tikzcd}
	V: & V_0 \arrow[r, "\delta_0", shift left] \arrow[r, "\delta_1"', shift right] & V_1 \arrow[r, "\delta_0", shift left=2] \arrow[r, "\delta_2"', shift right=2] \arrow[r, "\delta_1" description] & V_2 \arrow[r, shift left=3] \arrow[r, shift left] \arrow[r, shift right=3] \arrow[r, shift right] & \cdots
	\end{tikzcd}
\end{center}
is the DG-vector space
\[ \Tot(V)= \big\{ (x_n) \in \prod_{n\geq 0} A_n \otimes V_n \ | \ (\delta^*_k \otimes \Id )x_n = (\Id \otimes \delta_k) x_{n-1} \text{ for every } 0 \leq k \leq n \big\} ,\]
with differential induced by the one on $\prod_{n\geq 0} A_n \otimes V_n$. To simplify notation, we will denote this differential by $d_{\Tot}= d_{A}+ d_{V}$, where $d_A$ denotes the differential of polynomial differential forms, and $d_V$ the differential on $V$.

If $f\colon V \to W$ is a morphism of semicosimplicial DG-vector spaces, then $\Tot(f) \colon \Tot(V) \to \Tot(W)$ is defined as the restriction of the map
\[ \prod \Id \otimes f \colon \prod_{n\geq 0} A_n \otimes V_n \to \prod_{n\geq 0} A_n \otimes W_n .\]
\end{definition}

The $\Tot$ functor is exact (see e.g. \cite{Rational, Nav}): given semicosimplicial DG-vector spaces $V,W, Z$ and morphisms
$f\colon V\to W$, $g \colon W \to Z$ such that for every $n \geq 0$ the sequence
\begin{center}
	\begin{tikzcd}
	0 \arrow[r] & V_n \arrow[r, "f"] & W_n \arrow[r, "g"] & Z_n \arrow[r] & 0
	\end{tikzcd}
\end{center}
is exact, one obtains an exact sequence
\begin{center}
	\begin{tikzcd}
	0 \arrow[r] & \Tot(V) \arrow[r, "f"] & \Tot(W) \arrow[r, "g"] & \Tot(Z) \arrow[r] & 0.
	\end{tikzcd}
\end{center}

 Given two semicosimplicial DG-vector spaces $V$ and $W$, then $\Tot(V \times W)$ is naturally isomorphic to $\Tot(V)\times \Tot(W)$. An important consequence is the preservation of multiplicative structures; in particular, we will use the fact that the functor $\Tot$ sends semicosimplicial DG-Lie algebras 
to DG-Lie algebras.
\begin{example}\label{ex.cech}
	Let $\mathcal{E}$ be a finite complex of quasi-coherent sheaves on $X$, and $\mathcal{U}= \{ U_i\}$  an open cover of $X$. Denote by $U_{i_1 \cdots i_n}= U_{i_1}\cap \cdots \cap U_{i_n}$, and  consider the semicosimplicial DG-vector space of \v{C}ech cochains:
	\begin{center}
		\begin{tikzcd}
		\mathcal{E}(\mathcal{U}): \quad \prod\limits_i \mathcal{E}(U_i) \arrow[r, shift left, "\delta_0"] \arrow[r, shift right, "\delta_1"'] & {\prod\limits_{i,j}\mathcal{E}(U_{ij})} \arrow[r, shift left=2, "\delta_0"] \arrow[r, shift right=2, "\delta_2"'] \arrow[r, "\delta_1" description] & {\prod\limits_{i,j,k}\mathcal{E}(U_{ijk})} \arrow[r, shift left=3] \arrow[r, shift left] \arrow[r, shift right=3] \arrow[r, shift right] & \cdots.
		\end{tikzcd}
	\end{center}
	
	The Whitney integration theorem states that there exists a quasi-isomorphism between the Tot complex $\Tot(\mathcal{U},\mathcal{E})$ and the complex of \v{C}ech cochains $C^*(\sU, \mathcal{E})$ of  $\mathcal{E}$ (see \cite{Whi} for the $C^\infty$ version, \cite{Dup,  getzler04,luigi,Nav} for the algebraic version used here). Hence
	if the open cover $\mathcal{U}$ is affine the cohomology of $\Tot(\mathcal{U},\mathcal{E})$ is isomorphic to the hypercohomology of the complex of sheaves $\mathcal{E}$. If the open cover is affine the complex $\Tot(\mathcal{U},\mathcal{E})$ is a model for the module $\mathbf{R}\Gamma(X, \sE)$ of derived global sections  of $\sE$, see e.g. \cite{semicos,Meazz}.
	
	Moreover there is a canonical inclusion of the global sections of $\mathcal{E}$ in the totalisation $\Tot(\mathcal{U}, \mathcal{E})$: in fact, the injection
	\[\iota \colon \Gamma(X, \mathcal{E}) \to \prod\limits_i \mathcal{E}(U_i)\] is
 such that $\delta_0 \iota= \delta_1 \iota $ and therefore for every $a \in \Gamma(X, \mathcal{E})$ $$(1 \otimes \iota(a), 1 \otimes \delta_0 \iota (a), 1 \otimes \delta_0^2 \iota(a), \cdots)$$ belongs to $\Tot(\mathcal{U}, \mathcal{E})$. 
It is easy to see that if the complex $\mathcal{E}$ has trivial differential, every global section gives a cocycle in $\Tot(\sU, \mathcal{E})$.
\end{example}

\medspace

We are now ready to define connections on $\LL= \Ker \rho$, the kernel of the anchor map of a transitive DG-Lie algebroid $(\mathcal{A}, \rho)$ over $X$. Assume that $\LL$ is 
a finite complex of locally free sheaves and  
 fix an affine open cover $\mathcal{U}= \{ U_i\}$ of $X$. The short exact sequence 
 \begin{equation*}
 \begin{tikzcd}
 0 \arrow[r] & \Omega_X^1[-1]\otimes \LL \arrow[r, "\Id \otimes i"] & \Omega_X^1[-1] \otimes\mathcal{A} \arrow[r, "\Id \otimes \rho"] & \Omega_X^1[-1] \otimes \Theta_X \arrow[r] & 0.
 \end{tikzcd}
 \end{equation*}
 gives a short exact sequence of the corresponding semicosimplicial complexes of \v{C}ech cochains, and since the functor Tot is exact there is an exact sequence
 \begin{center}
 	\begin{tikzcd}[]
 	0 \arrow[r] & \Tot(\mathcal{U},  \Omega_X^1[-1]\otimes\LL ) \arrow[r, "\Id\otimes i"] & \Tot(\mathcal{U},\Omega_X^1[-1]\otimes  \mathcal{A} ) \arrow[r, "\Id\otimes \rho"] & \Tot(\mathcal{U}, \Omega_X^1[-1]\otimes \Theta_X) \arrow[r] & 0.
 	\end{tikzcd}
 \end{center}

 Denote by $d$ the differential on $\mathcal{A}$ and $\LL$, which can be extended to $\Omega_X^1[-1] \otimes \mathcal{A} $ and to $\Omega_X^1 [-1]\otimes \LL$ by setting $$d (\eta \otimes x) = (-1)^{\overline{\eta}} \eta \otimes d x = - \eta \otimes dx.$$
Denote by $d_{\Tot}$ the differentials on all the above $\Tot$ complexes: for $\Tot(\mathcal{U},  \Omega_X^1[-1]\otimes\LL )$ and $\Tot(\mathcal{U},\Omega_X^1[-1]\otimes  \mathcal{A} )$ the differential $d_{\Tot}$ is equal to $d_A + d$, while for $\Tot(\mathcal{U}, \Omega_X^1[-1]\otimes \Theta_X)$ one has that $d_{\Tot}$ is just $d_A$, where $d_A$ is the differential of polynomial differential forms on the affine simplex, see Definition~\ref{def.tot}.

 Because of the natural inclusion of global sections in the totalisation remarked upon in Example~\ref{ex.cech}, $\Id_{\Omega^1}$ belongs to $\Tot(\mathcal{U},   \Omega_X^1[-1]\otimes\Theta_X)$, where it has degree one.  
 \begin{definition}\label{def.conn-simpl}
 	A \textbf{simplicial lifting of the identity}  is an element $D $ of $\Tot(\mathcal{U},   \Omega_X^1[-1]\otimes \mathcal{A})$ such that $(\Id\otimes \rho) (D) = \Id_{\Omega^1} $ in $  \Tot(\mathcal{U}, \Omega_X^1[-1]\otimes \Theta_X)$.
 \end{definition}
It is clear that a simplicial lifting of the identity always exists and that $D$ has degree one in $\Tot(\sU, \Omega_X^1[-1] \otimes \mathcal{A})$.

\begin{remark}\label{rem.rho-deRham}
	Notice that 
	via the isomorphism $\Omega_X^1[-1] \otimes \Theta_X = \Omega_X^1[-1] \otimes \DER_{\K} (\Oh_X, \Oh_X) \cong \DER^*_{\K}(\Oh_X, \Omega_X^1 [-1])$, 
	we have that $$(\Id \otimes \rho)(D) = \dr \in \Tot(\sU, \DER^*_{\K}(\Oh_X, \Omega_X^1 [-1])).$$
\end{remark}

\medspace
 
 In order to define a connection on $\LL$, it is necessary to define a Lie bracket $$[-,-]\colon \Tot(\sU, \Omega_X^1[-1]\otimes \mathcal{A}) \times  \Tot(\sU, \LL) \to \Tot(\sU,\Omega_X^{1}[-1] \otimes \LL),$$ induced by the bracket of the following lemma.

\begin{lemma}\label{lem.operatoreagg'}
	There exists a well defined $\K$-bilinear bracket
	\[[-,-] \colon  (\Omega_X^1[-1]\otimes \mathcal{A}) \times \LL \to \Omega_X^1[-1] \otimes \LL.\]
\end{lemma}

\begin{proof}
	Denote by $i\colon \LL\to \mathcal{A}$ the inclusion, take $\eta \otimes a$ with $\eta \in \Omega_X^1[-1]$ and $a \in \mathcal{A}$, and define for $x \in \LL$
	\[ [\eta \otimes a, x]:= \eta \otimes [a, i(x)].\]
	Notice that the Leibniz identity in Definition~\ref{def.DGLiealg} implies that
	\[ [fa_1, a_2] = f[a_1,a_2] - (-1)^{\overline{a_1}\ \overline{a_2}} \rho (a_2) (f) a_1.\]
	Hence the bracket $[\eta \otimes a, x]$ is well defined: for any $f \in \Oh_X$
	\[ [\eta \otimes f a , x ] = \eta \otimes [fa, x] = \eta \otimes \big(f [a, x] - (-1)^{\overline{a}\ \overline{x}} \rho (x) (f) a\big) = \eta \otimes f [a, x] = f \eta \otimes [a, x ] = [f \eta \otimes a, x].\]
	It is  clear that $[\eta \otimes a, x]$ belongs to $\Omega^1_X[-1]\otimes \LL$:
	\[ (\Id \otimes \rho) ([\eta \otimes a, x]) = (\Id \otimes \rho) (\eta \otimes [a,x]) = \eta \otimes [\rho (a), \rho(x)]=0.\]
\end{proof}

Since the functor $\Tot$ preserves products, the map  $$ [-,-] \colon (\Omega_X^1[-1] \otimes \mathcal{A}) \times \LL \to \Omega^1_X[-1] \otimes \LL$$ induces a $\K$-bilinear map
\begin{equation}\label{eq.BracketTot}
 [-,-]\colon \Tot(\sU, \Omega_X^1[-1] \otimes \mathcal{A}) \times  \Tot(\sU, \LL) \to \Tot(\sU,\Omega_X^{1}[-1] \otimes \LL), 
\end{equation}
which is defined componentwise as the restriction of
	\begin{align*} A_n \otimes \prod& (\Omega^1_X[-1] \otimes \mathcal{A} )(U_{i_1 \cdots i_n}) \times  A_n \otimes \prod  \mathcal{L} (U_{i_1 \cdots i_n}) \xrightarrow{[-,-] }  A_n \otimes \prod (\Omega^1_X[-1] \otimes \LL )(U_{i_1 \cdots i_n})\\
&[\eta_n\otimes (t_{i_1\cdots i_n}), \phi_n \otimes (u_{i_1\cdots i_n}) ] = \eta_n \phi_n \otimes ([(-1)^{\overline{\phi_n}\ \overline{t_{i_1\cdots i_n}}}t_{i_1\cdots i_n}, u_{i_1\cdots i_n}]),
	\end{align*}
	for $\eta_n, \phi_n$ in $A_n$, $t_{i_1\cdots i_n}$ in $ (\Omega^1_X[-1] \otimes \mathcal{A} )(U_{i_1 \cdots i_n}) $ and $u_{i_1\cdots i_n}$ in $\mathcal{L} (U_{i_1 \cdots i_n})$.

	\begin{definition}\label{def.connessione}
		A \textbf{connection} on $\LL$ is the adjoint operator of a simplicial lifting of the identity $D \in  \Tot(\mathcal{U},  \Omega_X^1[-1]\otimes\mathcal{A}  )$ 
		\begin{equation*}
		\nabla =[D, -]\colon  \Tot(\mathcal{U},  \LL  ) \to \Tot(\mathcal{U},   \Omega_X^{1}[-1]\otimes\LL ),
		\end{equation*}
		where $[-,-] \colon  \Tot(\sU, \Omega_X^1[-1] \otimes \mathcal{A}) \times  \Tot(\sU, \LL) \to \Tot(\sU,\Omega_X^{1}[-1] \otimes \LL)$ is the bracket in \eqref{eq.BracketTot}.
		It is a $\K$-linear operator.
	\end{definition}

\medspace

We will now examine the relationship between connections and particular representatives of extension classes. 
The short exact sequence 
\begin{equation}\label{eq.succEs}
	\begin{tikzcd}
	0 \arrow[r] & \LL  \arrow[r] & \mathcal{A} \arrow[r, "\rho"] & \Theta_X \arrow[r] & 0
	\end{tikzcd}
\end{equation}
gives an {extension class} $[u_\rho ] \in \Ext_X^1(\Theta_X, \LL)$. 
It is possible to give a representative of $[u_\rho]$ in the totalisation $\Tot(\sU, \Omega_X^1[-1] \otimes \LL)$ with respect to an affine open cover $\mathcal{U}$ of $X$.

\begin{definition}\label{def.Atiyah-cocycle}
	An \textbf{extension cocycle} $u$ of the transitive DG-Lie algebroid $\mathcal{A}$ is the differential of a simplicial lifting of the identity $D$ in $\Tot(\mathcal{U},   \Omega_X^1[-1]\otimes \mathcal{A})$, $u=d_{\Tot} D$.
\end{definition}
Notice that $ u $ belongs to $ \Tot (\mathcal{U}, \Omega_X^1[-1]\otimes \LL )$: 
$$(\Id\otimes \rho) u = (\Id\otimes \rho) d_{\Tot}(D) = d_{\Tot} (\Id\otimes \rho) D = d_{\Tot} \Id_{\Omega^1}=0,$$ where the last equality is a consequence of the fact that
$\Id_{\Omega^1}$ is a global section and 
$\Omega_X^1[-1] \otimes \Theta_X$ has trivial differential {(see Example~\ref{ex.cech})}. 
Note that  $u$ has degree two in $ \Tot (\mathcal{U}, \Omega_X^1[-1]\otimes \LL )$ and that $d_{\Tot} u = d_{\Tot} d_{\Tot} D =0$.

Using the isomorphisms
$\Omega_X^1[-1]\otimes \LL \cong \HOM^*_{\Oh_X}(\Theta_X[1], \LL) \cong \HOM^*_{\Oh_X}(\Theta_X, \LL)[-1],$
the cohomology class of $u$ belongs to 
\begin{align*}
H^2(\Tot(\mathcal{U}, \Omega_X^1[-1]\otimes \LL)) &\cong   H^2(\Tot(\sU, \HOM^*_{\Oh_X}(\Theta_X, \LL)[-1])) \cong \mathbb{H}^2(X, \HOM^*_{\Oh_X}(\Theta_X, \LL)[-1]) \\
&\cong \mathbb{H}^1(X, \HOM^*_{\Oh_X}(\Theta_X, \LL))  \cong \Ext_X^1(\Theta_X, \LL).
\end{align*}
This cohomology class does not depend on the chosen simplicial lifting of the identity: if $D$ and $D'$ are two simplicial liftings of the identity in $\Tot(\sU, \Omega_X^1[-1]\otimes \mathcal{A})$, we have that $(\Id \otimes \rho)(D-D')= \Id_{\Omega^1}- \Id_{\Omega^1}=0$, so  $D-D'$ belongs to $ \Tot(\sU, \Omega_X^1 [-1]\otimes \LL)$ and
$d_{\Tot}D $ and $d_{\Tot}D'$ differ by the coboundary $d_{\Tot}(D'-D)$.
It is easy to see that the cohomology class of $u$ is trivial if and only if the short exact sequence in \eqref{eq.succEs} splits.

\begin{lemma}\label{lem.diff e nabla}
	Let  $\na \colon \Tot(\sU, \LL) \to \Tot(\mathcal{U},  \Omega_X^1[-1]\otimes\mathcal{L}  )$ be a connection on $\LL$, associated to the simplicial lifting of the identity $D \in \Tot(\mathcal{U}, \Omega_X^1[-1] \otimes \mathcal{A})$. Let $u=d_{\Tot} D$ be the corresponding extension cocycle, then for every $x$ in $\Tot(\sU, \LL)$ we have that 
	\begin{equation*}
	\na(d_{\Tot} x) = [u,x ]- d_{\Tot} \na (x).
	\end{equation*}
\end{lemma}

\begin{proof}
	Recall that $d$ denotes the differential of $\mathcal{A}$ and $\LL$, which can be extended to $\Omega_X^1[-1] \otimes \mathcal{A} $ and to $\Omega_X^1[-1] \otimes \LL$ by setting $d (\eta \otimes x) = (-1)^{\overline{\eta}} \eta \otimes d x = -\eta \otimes d x.$
	It is easy to see that for the $\K$-bilinear map  $ [-,-] \colon (\Omega_X^1[-1] \otimes \mathcal{A}) \times \LL \to \Omega^1_X[-1] \otimes \LL$ of Lemma~\ref{lem.operatoreagg'},
	\begin{equation*}
	d[\eta \otimes a, x]= [d(\eta \otimes a), x] + (-1)^{\overline{\eta} + \overline{a}} [\eta \otimes a, dx].
	\end{equation*}
	A straightforward calculation then shows that for $z \in \Tot(\sU, \Omega_X^1[-1] \otimes \mathcal{A})$ and $w \in \Tot(\sU, \LL)$,  \[d_{\Tot} [z,w]= [d_{\Tot} z, w] + (-1)^{\overline{z}} [z, d_{\Tot}w],\]
	and the conclusion follows from the fact $u = d_{\Tot}D$.
\end{proof}

\section{Cyclic forms and $L_\infty$ morphisms}\label{sec.cyclic}

This section describes cyclic forms on DG-Lie algebroids and illustrates how DG-Lie algebroid representations give rise to cyclic forms. We then discuss induced  
cyclic forms on the Thom-Whitney totalisation and the property of $d_{\Tot}$-closure. 
 The central result is the construction of a $L_\infty$ morphism associated to a connection and to a $d_{\Tot}$-closed cyclic form for a transitive DG-Lie algebroid.
   This allows us to state the results of \cite{linfsemireg} for a coherent sheaf admitting a finite locally free resolution on a smooth separated scheme of finite type over a field $\K$ of characteristic zero.

Let $\mathcal{A}$ be a DG-Lie algebroid over a smooth separated scheme $X$ of finite type over a field $\K$ of characteristic zero, with anchor map $\rho \colon \mathcal{A} \to \Theta_X$.  Assume that the kernel of the anchor map $\LL$ is a finite complex of locally free sheaves. Notice that for any $a \in \mathcal{A}$ and $x \in \LL$, the bracket 
$[a,x]$ belongs to $\LL$:
\[ \rho ([a,x]) = [\rho (a), \rho(x)]=0.\]
\begin{definition}
	A cyclic bilinear form on a DG-Lie algebroid $(\mathcal{A},\rho)$ is a graded symmetric $\Oh_X$-bilinear product of degree zero on $\LL= \Ker \rho$
	\[\la -, - \ra \colon \LL \times \LL \to \Oh_X \]
	such that for all sections $x,y$ of $\LL$ and $a $ of $ \mathcal{A}$
	\[ \la [a,x], y \ra  + (-1)^{\overline{a}\ \overline{x}} \la x, [a, y] \ra = \rho (a)(\la x, y \ra).\]
\end{definition}

Notice that the definition implies that for all $x,y,z \in \LL$
\[ \la x, [y,z] \ra =  \la [x,y], z \ra.\]

In the following two examples, the cyclicity of the forms will follow from  Lemma~\ref{lem.compat}.
\begin{example}\label{ex.ad fasci}
	An example of cyclic form on $(\mathcal{A},\rho)$ is induced by the Killing form. Consider the adjoint representation as a morphism of sheaves of DG-Lie algebras 
	\[ \ad \colon \LL \to \HOM^*_{\Oh_X}(\LL, \LL),\quad a \mapsto [a,-]\]
	and consider the trace map $\Tr \colon \HOM^*_{\Oh_X}(\LL, \LL) \to \Oh_X$, which is morphism of sheaves of DG-Lie algebras (when considering $\Oh_X$ as a trivial sheaf of DG-Lie algebras).  Then one can define the form 
	\[ \langle -, - \rangle \colon \LL \times \LL  \to \Oh_X, \quad (x,y) \mapsto \Tr(\ad x \ad y).\]

\end{example}

	\begin{example}\label{ex.traccia}
		For the DG-Lie algebroid $\sD^*(X, \sE)$ of Example~\ref{ex.coppie}, 
		\begin{center}
			\begin{tikzcd}
			0 \arrow[r] &\HOM^*_{\Oh_X}(\sE, \sE) \arrow[r] &\sD^*(X,\sE) \arrow[r] &\Theta_X \arrow[r] &0,
			\end{tikzcd}
		\end{center}a natural bilinear form on $\HOM^*_{\Oh_X}(\sE, \sE)$ is induced by the trace map $\Tr \colon \HOM^*_{\Oh_X}(\sE, \sE) \to \Oh_X$ as follows:
	\[ \HOM^*_{\Oh_X}(\sE, \sE) \times \HOM^*_{\Oh_X}(\sE, \sE) \to \Oh_X, \quad \la f, g\ra := -\Tr (fg). \]
	\end{example}

\medspace
	
	The Leibniz identity of Definition~\ref{def.DGLiealg}
	\[ [a, fx] = f[a,x]+ \rho(a)(f) x \quad \forall a \in \mathcal{A},\ x \in \LL,\ f \in \Oh_X \]
	 can be restated by noticing 
	 that for all $a$ in $\mathcal{A}$ the operator $(\rho(a), [a,-])$ belongs to $\sD^*(X, \LL)$ of Example~\ref{ex.coppie}. Hence there is 
	a morphism of DG-Lie algebroids
	\[ \ad \colon \mathcal{A} \to \sD^*(X, \LL).\]
The morphism $\ad \colon \mathcal{A} \to \sD^*(X, \LL)$  restricts to the morphism $\ad \colon \LL \to \HOM^*_{\Oh_X}(\LL, \LL)$ of Example~\ref{ex.ad fasci}, so that the following diagram commutes
\begin{center}
	\begin{tikzcd}
		0 \arrow[r] & \LL \arrow[r, "i"] \arrow[d, "\ad"] & \mathcal{A} \arrow[r, "\rho"]\arrow[d, "\ad"] & \Theta_X \arrow[r] \arrow[d, double, no head]  &0\\
		0 \arrow[r] & \HOM^*_{\Oh_X}(\LL, \LL) \arrow[r] & \sD^*(X,\LL) \arrow[r, "\alpha"] & \Theta_X \arrow[r] &0.
	\end{tikzcd}
\end{center}
This motivates the following definition:
\begin{definition}
	A representation of a DG-Lie algebroid $(\mathcal{A}, \rho)$ over $X$ is a morphism of DG-Lie algebroids
	$\theta \colon \mathcal{A} \to \sD^*(X, \sE)$, where $\sE$ is a finite complex of locally free sheaves over $X$:
	\begin{center}
		\begin{tikzcd}
		\mathcal{A} \arrow[rd, "\rho"'] \arrow[rr, "\theta"] &          & {\sD^*(X, \mathcal{E})} \arrow[ld, "\alpha"] \\
		& \Theta_X .&                                            
		\end{tikzcd}
	\end{center}
\end{definition}

Every representation $\theta \colon \mathcal{A} \to \sD^*(X, \sE)$ induces a form
$\la -,- \ra_{\theta} \colon \LL \times \LL \to \Oh_X$:  for any $x \in \LL$ we have that $\alpha \circ \theta (x)= \rho (x)=0$, so that 
\[ \theta|_{\LL} \colon \LL \to \HOM^*_{\Oh_X}(\sE,\sE),\] and using the trace map
$\Tr \colon \HOM^*_{\Oh_X}(\sE,\sE) \to \Oh_X$
we can define 
for $x,y$ sections of $\LL$
\[ \la x, y \ra_{\theta} := \Tr (\theta(x) \theta(y)).\]
Forms obtained in this way are cyclic:

\begin{lemma}\label{lem.compat}
	For any DG-Lie algebroid representation $\theta \colon \mathcal{A} \to \sD^*(X, \sE)$ the induced form $ \langle -, - \rangle_\theta \colon  \LL  \times   \LL \to \Oh_X$ is cyclic.
\end{lemma}
\begin{proof}

	For $a \in \mathcal{A}$ and $x,y \in \LL$
	\begin{align*}
	\la [a,x], y \ra_{\theta} +  (-1)^{\overline{a}\ \overline{x}}\la x, [a,y]\ra_{\theta} &= \Tr (\theta ([a,x])\theta (y) +  (-1)^{\overline{a}\ \overline{x}}\theta (x)\theta([a,y]))\\
	& =\Tr ([\theta (a),\theta(x)]\theta (y) +  (-1)^{\overline{a}\ \overline{x}} \theta (x)[\theta(a), \theta(y)])\\
	&= \Tr (\theta(a)\theta(x)\theta(y) - (-1)^{\overline{a}(\overline{x}+ \overline{y})}\theta(x)\theta(y)\theta(a))\\
	& =\Tr ([\theta(a), \theta(x)\theta(y)]).
	\end{align*}
	 Notice that if $\overline{a} \neq 0$ then $a$ belongs to $\LL$, so that $\theta(a)$ belongs to $\HOM^*_{\Oh_X}(\sE, \sE)$ and it is clear that $\Tr ([\theta(a), \theta(x)\theta(y)])=0$, by the properties of the trace map.
	 
	The only remaining non-trivial case is when $\overline{a}= \overline{x}+ \overline{y}=0$. Let $\{e_i^k\}$ with $i=1, \cdots , n_k$ be a local basis of $\LL^k$ and let
	\[ \theta(a)(e_i^k)= \sum_j A^k_{ij}e_j^k,\quad \theta(x)\theta(y)(e_i^k)= \sum_j B^k_{ij}e_j^k,\quad A^k_{ij}, B^k_{ij} \in \Oh_X,\]
	then
	\begin{align*}
	[\theta(a), \theta(x)\theta(y)](e_i^k) &= \theta(a)\theta(x)\theta(y)(e_i^k) - \theta(x)\theta(y)\theta(a)(e_i^k) \\
	&=\theta(a)\big(\sum_j B^k_{ij}e_j^k\big)- \theta(x)\theta(y)\big(\sum_j A^k_{ij}e_j^k\big)\\
	&= \sum_j B^k_{ij} \theta(a)(e_j^k) + \sum_j (\alpha \circ \theta) (a)(B^k_{ij}) e_j^k - \sum_{j,s} A^k_{ij}B^k_{js}e_s^k\\
	& =\sum_{j,s} B^k_{ij} A^k_{js}e_s^k + \sum_j \rho (a)(B^k_{ij}) e_j^k - \sum_{j,s} A^k_{ij}B^k_{js}e_s^k.
	\end{align*}
	The trace of $[\theta(a), \theta(x)\theta(y)]$ is hence equal to
	\begin{align*} \sum_k (-1)^k \big( \sum_{j,i} B^k_{ij} A^k_{ji} + \sum_i \rho (a)(B^k_{ii})  - \sum_{j,i} A^k_{ij}B^k_{ji} \big)&= \sum_{k,i} (-1)^k \rho (a) (B^k_{ii})=\\
	\rho(a)\big( \sum_{k,i} (-1)^k B_{ii}^k \big) =\rho(a)\Tr(\theta(x)\theta(y))&= \rho (a) (\la x, y \ra_{\theta}).\end{align*}
\end{proof}

 For every $i \geq 0$, let $ \Omega_X^i[-i]$ denote the sheaf $\Omega_X^i$ seen as a trivial complex concentrated in degree $i$.
Any cyclic form $\langle -, - \rangle \colon \LL \times \LL  \to \Oh_X$ can be extended to a collection of $\Oh_X$-bilinear forms
\begin{equation*}
 \langle -, - \rangle \colon ( \Omega_X^i[-i] \otimes \LL ) \times  (\Omega_X^j[-j] \otimes \LL )\to \Omega_X^{i+j}[-i-j], \quad i,j \geq 0,
 \end{equation*}
according to the Koszul sign rule, by setting for $x,y \in \LL$, $\omega \in \Omega_X^i[-i]$, an $\eta \in \Omega_X^j[-j]$ 
$$\la  \omega  \otimes x,  \eta  \otimes y  \ra = (-1)^{\overline{x}j}  \omega \wedge \eta \la x , y \ra . $$
 It is immediate to see that this form  is cyclic in the sense that 
\[ \la [b, x], y\ra + (-1)^{\overline{b}\overline{x}} \la x, [b,y] \ra = (\Id \otimes \rho)(b) (\la x, y\ra) \quad \forall b \in \Omega_X^1[-1] \otimes \mathcal{A},\ \ \forall x,y \in  \mathcal{L},\]
 where the bracket is the one of Lemma~\ref{lem.operatoreagg'}, and the anchor map has been extended to $\Omega_X^1[-1]\otimes \mathcal{A}$ by setting $(\Id \otimes \rho) (\omega \otimes a) := \omega 	\otimes \rho (a)$.

\begin{definition}\label{def.prop-form}
	A cyclic form $ \langle -, - \rangle \colon  \LL \times  \LL  \to \Oh_X$ is
	\textbf{$d$-closed} if for all $z,w \in  \LL$
	\[ \langle dz, w \rangle + (-1)^{\overline{z}} \langle z, d w \rangle =0 .\]

\end{definition}

\begin{lemma}\label{lem.d-chiusura}
	For any DG-Lie algebroid representation $\theta \colon \mathcal{A} \to \sD^*(X, \sE)$ the induced cyclic form $ \langle -, - \rangle_\theta \colon   \LL  \times   \LL \to \Oh_X$ is $d$-closed.
\end{lemma}

\begin{proof}
Since  $ \theta|_{\LL} \colon \LL \to \HOM^*_{\Oh_X}(\sE,\sE)$ is a morphism of DG-Lie algebroids it commutes with differentials: for $x \in \LL$, $$\theta(dx)= d_{\HOM^*(\sE,\sE)} (\theta (x)) = [d_\sE, \theta (x)].$$
For $x,y$ sections of $\LL$
\begin{align*}
&\la dx, y \ra_{\theta}  + (-1)^{\overline{x}}\la x, dy \ra_{\theta} = \\
&\Tr (\theta(dx) \theta(y )+ (-1)^{\overline{x}} \theta(x)\theta(dy))= \Tr ([d_\sE, \theta (x)] \theta(y )+ (-1)^{\overline{x}} \theta(x)[d_\sE, \theta (y)]) = \\
&\Tr (d_\sE \theta(x)\theta(y)- (-1)^{\overline{x}+\overline{y}} \theta(x)\theta(y)d_\sE)= \Tr ([d_\sE, \theta(x)\theta(y)])=0.
\end{align*}
\end{proof}

It follows from the properties of the Thom-Whitney totalisation  functor $\Tot$ that
	every collection of cyclic forms $\la-,-\ra\colon (\Omega^i_X[-i] \otimes \LL) \times (\Omega^j_X[-j] \otimes \LL) \to \Omega^{i+j}_X[-i-j]$, with $i,j \geq 0$, induces a collection of $\K$-bilinear forms
	\[ \la-,-\ra \colon \Tot(\mathcal{U}, \Omega^i_X[-i] \otimes \LL) \times \Tot(\mathcal{U}, \Omega^j_X[-j] \otimes \LL) \to \Tot(\mathcal{U}, \Omega^{i+j}_X[-i-j]).\]
	Recalling Definition~\ref{def.tot}, 
	the required forms are induced componentwise by the restriction of $$ A_n \otimes \prod_{i_1 \cdots i_n} (\Omega^i_X[-i] \otimes \LL )(U_{i_1 \cdots i_n}) \times A_n \otimes \prod_{i_1 \cdots i_n} (\Omega^j_X[-j] \otimes \LL )(U_{i_1 \cdots i_n}) \to  A_n \otimes \prod_{i_1 \cdots i_n} \Omega^{i+j}_X[-i-j] (U_{i_1 \cdots i_n}),$$
	\[ \la \eta_n \otimes (x_{i_1 \cdots i_n}), \omega_n \otimes (y_{i_1 \cdots i_n}) \ra :=    \eta_n \omega_n ( \la (-1)^{\overline{\omega_n}\ \overline{(x_{i_1 \cdots i_n})}} x_{i_1 \cdots i_n}, y_{i_1 \cdots i_n} \ra),\]
	with $\eta_n, \omega_n$ in $A_n$, $x_{i_1 \cdots i_n}$ in $ (\Omega^i_X[-i] \otimes \LL )(U_{i_1 \cdots i_n})$ and $ y_{i_1 \cdots i_n}$ in  $ (\Omega^j_X[-j] \otimes \LL )(U_{i_1 \cdots i_n})$.
	
	\medskip
	
	Let $(\Omega_X^* = \oplus_p \Omega_X^p[-p], \dr)$ denote the de Rham complex.
	 In the following, when working with $\Tot(\mathcal{U}, \Omega_X^*)$  
	the differential is denoted by $d_{\Tot}$ if $\Omega_X^*= \oplus_p \Omega_X^p[-p]$  is considered as complex with trivial differential, and by $d_{\Tot} + \dr$ if  it is considered as a complex with the de Rham differential.

	\begin{lemma}\label{lem.formTot}
		The form induced on the totalisation by a cyclic form $\la-,-\ra\colon  \LL\times  \LL\to \Oh_X$ is  cyclic: for all $b \in \Tot(\sU,  \Omega^1_X[-1] \otimes\mathcal{A})$ and $z,w \in \Tot(\sU,  \mathcal{L})$ one has that
		\[ \la [b,z], w\ra + (-1)^{\overline{b} \overline{z}} \la z, [b,w] \ra = (\Id \otimes \rho)(b) (\la z,w \ra).\]

		Moreover if the form $\la-,-\ra\colon  \LL \times  \LL\to \Oh_X$ is $d$-closed  (Definition~\ref{def.prop-form}), then for the induced form, for
		 $z \in \Tot(\mathcal{U}, \Omega^i_X[-i]\otimes \LL)$  and $w \in \Tot(\mathcal{U}, \Omega^j_X[-j]\otimes \LL)$ we have that
		\[ \la d_{\Tot}z, w\ra + (-1)^{\overline{z}}\la z, d_{\Tot} w \ra = d_{\Tot} \la z,w \ra.\]
	The above condition will be called \textbf{$d_{\Tot}$-closure}.

	\end{lemma}
	
	\begin{proof}
 		For the first item, since everything is defined componentwise, it suffices to prove that for  every $a \in A_n \otimes \prod  (\Omega_X^1[-1]\otimes \mathcal{A})(U_{i_1 \cdots i_n})$ and every $x, y \in A_n \otimes \prod \LL (U_{i_1 \cdots i_n})$ 
		\[\la [a, x], y \ra + (-1)^{\overline {a}\ \overline{x}}\la x, [a, y]\ra =  (\Id \otimes \rho)(a)(\la x, y \ra)\] for every $n\geq 0$ .
			By linearity let $a =  \omega_{n} \otimes z_n $, with $\omega_{n} \in A_n$ and $z_n \in \prod (\Omega_X^1[-1]\otimes \mathcal{A}) (U_{i_1 \cdots i_n})$;
		let $x = \eta_n \otimes x_n$ and $y = \phi_n \otimes y_n$, with $\eta_n, \phi_n $ in $A_n$ and $ x_n, y_n$  in $\prod \LL (U_{i_1 \cdots i_n})$.
		Then
		\begin{align*}
		&\la [a, x], y \ra + (-1)^{\overline{a}\ \overline{x}}\la x, [a, y]\ra = \\
		&  \la [ \omega_{n} \otimes z_n,\eta_n \otimes x_n],\phi_n \otimes y_n  \ra +(-1)^{\overline{a}\ \overline{x}} \la \eta_n \otimes x_n, [ \omega_{n} \otimes z_n, \phi_n \otimes y_n]\ra = \\
		& (-1)^{\overline{\eta_n}\ \overline{z_n}}\la \omega_n \eta_n \otimes [z_n, x_n] , \phi_n \otimes y_n \ra
		+ (-1)^{\overline{a}\ \overline{x}} \la \eta_n \otimes x_n, (-1)^{\overline{\phi_n} \overline{z_n}} \omega_{n}\phi_n \otimes [z_n, y_n]  \ra =\\
		& (-1)^{\overline{\phi_n} (\overline{z_n} +\overline{x_n}) + \overline{\eta_n}\ \overline{z_n}} \big(\omega_{n} \eta_n \phi_n   \la [z_n, x_n], y_n \ra  +(-1)^{\overline{\omega_n}\ \overline{\eta_n}+  \overline{z_n}\ \overline{x_n}} \eta_n \omega_n \phi_n  \la x_n, [z_n, y_n] \ra \big) = \\
		& (-1)^{\overline{\phi_n}(\overline{z_n} + \overline{x_n}) + \overline{\eta_n}\ \overline{z_n}}\omega_n \eta_n \phi_n  \big( \la [z_n, x_n], y_n \ra +(-1)^{\overline{x_n}\ \overline{z_n}} \la x_n, [z_n, y_n]\ra  \big)  =\\
		&  (-1)^{\overline{\phi_n}(\overline{z_n} + \overline{x_n}) + \overline{\eta_n}\ \overline{z_n}}\omega_n \eta_n \phi_n    (\Id \otimes\rho) (z_n)( \la x_n, y_n\ra)  =  (\Id \otimes\rho) (a)(\la x, y \ra) .
		\end{align*}

		For the second item, recall that $d_{\Tot}$ is the differential on  $\Tot(\mathcal{U}, \Omega^*_X )$ when considering $\Omega^*_X$ as a complex with trivial differential. Again, since everything is defined componentwise, it is sufficient to prove that
		\[ \la d_{\Tot}(\eta_n \otimes x_n), \omega_n \otimes y_n \ra + (-1)^{\overline{x_n}+ \overline{\eta_n}}\la \eta_n \otimes x_n, d_{\Tot} ( \omega_n \otimes y_n) \ra = d_{\Tot} \la \eta_n \otimes x_n, \omega_n \otimes y_n \ra,\]
	for $\eta_n, \omega_n \in A_n$ and $x_n \in \prod(\Omega^i_X[-i] \otimes \LL )(U_{i_1 \cdots i_n})$ and $ y_n\in \prod(\Omega^j_X[-j] \otimes \LL )(U_{i_1 \cdots i_n})$. Then 
		\begin{align*}
		&\la d_{\Tot}(\eta_n \otimes x_n), \omega_n \otimes y_n \ra + (-1)^{\overline{x_n}+ \overline{\eta_n}}\la \eta_n \otimes x_n, d_{\Tot} ( \omega_n \otimes y_n) \ra=\\
		& \la d_{A_n}\eta_n \otimes x_n + (-1)^{\overline{\eta_n}} \eta_n \otimes dx_n, \omega_n \otimes y_n \ra + (-1)^{\overline{x_n}+ \overline{\eta_n}}\la \eta_n \otimes x_n, d_{A_n}  \omega_n \otimes y_n + (-1)^{\overline{\omega_n}}  \omega_n \otimes dy_n\ra\\
		&=(-1)^{\overline{\omega_n}\ \overline{x_n}} d_{A_n}(\eta_n) \omega_n \la x_n, y_n \ra + (-1)^{\overline{\eta_n} + \overline{\omega_n} (\overline{x_n} + 1)} \eta_n \omega_n \la dx_n, y_n \ra  + \\
		&+(-1)^{\overline{x_n}+ \overline{\eta_n}+ (\overline{\omega_n} +1) \overline{x_n}} \eta_n d_{A_n} (\omega_n) \la x_n, y_n \ra  + (-1)^{\overline{x_n}+ \overline{\eta_n} + \overline{\omega_n} +\overline{\omega_n}\ \overline{x_n}} \eta_n \omega_n \la x_n, dy_n \ra \\
		&= (-1)^{\overline{\omega_n}\ \overline{x_n}} d_{A_n} (\eta_n \omega_n) \la x_n, y_n \ra + (-1)^{\overline{\eta_n} + \overline{\omega_n} (\overline{x_n} + 1)} \eta_n \omega_n \big( \la dx_n, y_n \ra + (-1)^{\overline{x_n}} \la x_n, dy_n \ra  \big )\\
		&=(-1)^{\overline{\omega_n}\ \overline{x_n}} d_{A_n} (\eta_n \omega_n) \la x_n, y_n \ra = d_{A_n} \la \eta_n \otimes x_n, \omega_n \otimes y_n \ra = d_{\Tot} \la \eta_n \otimes x_n, \omega_n \otimes y_n \ra,
		\end{align*}
		where $d_{A_n}$ denotes the differential on $A_n$,  the differential graded algebra of polynomial differential forms on the affine $n$-simplex.

	\end{proof}

	\begin{corollary}\label{cor.nabladeRham}
			Let $D$ in $\Tot(\sU, \Omega_X^1 [-1]\otimes \mathcal{A})$ be a simplicial lifting of the identity and let $$\na= [D, -,] \colon \Tot(\sU, \LL) \to   \Tot(\sU, \Omega_X^1[-1]\otimes \LL) $$ be its associated connection, as in Defintion~\ref{def.connessione}. Then for any cyclic form
			$$\la-,-\ra \colon \Tot(\mathcal{U}, \Omega^i_X[-i] \otimes \LL) \times \Tot(\mathcal{U}, \Omega^j_X[-j] \otimes \LL) \to \Tot(\mathcal{U}, \Omega^{i+j}_X[-i-j] ),\quad i,j \geq 0,$$ we have 
		\begin{equation*}
		\la \na (x), y \ra + (-1)^{\overline{x}} \la x, \na (y) \ra= \dr \la x, y \ra
		\end{equation*}
		for $x,y \in \Tot(\mathcal{U},  \LL) $.
	\end{corollary}
\begin{proof}
	It follows from the cyclicity of the form and by Remark~\ref{rem.rho-deRham}.
\end{proof}

\medspace

The next part is dedicated to defining an  $L_\infty$ morphism associated to a connection and to a $d_{\Tot}$-closed cyclic form on a transitive DG-Lie algebroid.
We assume that the reader is familiar with the notions and basic properties of DG-Lie algebras and
$L_{\infty}$ morphisms between them; details can be found in \cite{getzler04,K,ManRendiconti,LMDT} and in the references therein. The definition of an
$L_{\infty}$ morphism between a DG-Lie algebra and an abelian DG-Lie algebra, i.e., a DG-Lie algebra with trivial bracket is recalled here, because it will be 
needed for explicit calculations.

Let $V$ be a graded vector space over a field of characteristic 0. Let $\sigma$ be a permutation of $\{1,\ldots,n\}$, and let $v_1,\ldots,v_n$ be
homogeneous  vectors  of $V$: denote by
$\chi(\sigma;v_1,\ldots,v_n)=\pm 1$ the antisymmetric Koszul sign, defined by the relation 
\[ v_{\sigma(1)}\wedge\cdots\wedge v_{\sigma(n)}=\chi(\sigma;v_1,\ldots,v_n)\,  
v_{1}\wedge\cdots\wedge v_{n}\,\]
in the $n$th exterior power $V^{\wedge n}$. If the vectors $v_1,\ldots,v_n$ are clear from the context we will write $\chi(\sigma)$ instead of 
$\chi(\sigma;v_1,\ldots,v_n)$. Given two non-negative integers $p$ and $q$,  $S(p,q)$  denotes the set of $(p,q)$-shuffles, the permutations $\sigma$ of the set $\{1, \cdots, p+q\}$ such that 
\[ \sigma(1) < \sigma(2) <\cdots < \sigma(p); \quad \sigma(p+1) < \cdots < \sigma(p+q).\] Recall that the cardinality of $S(p,q)$ is $\binom{p+q}{p}$.
Because of the universal property of wedge powers, every linear map
$V^{\wedge p}\to W$ will be interpreted as a graded skew-symmetric $p$-linear map $V\times\cdots\times V\to W$. 

	\begin{definition}\label{def.morfismo}
	Let $(V,\delta,[-,-])$ be a DG-Lie algebra and $(M,d)$ an abelian DG-Lie algebra. An $L_{\infty}$ morphism $g\colon V\to M$ is a sequence of maps 
	$g_n\colon V^{\wedge n}\to M$, $n\ge 1$, with $g_n$ of degree $1-n$ such that, for every $n$ and every 
	$v_1,\ldots,v_n\in V$ homogeneous, the following conditions $C_i$ are satisfied for all $i \in \N$:
	\[ (C_1):\quad  dg_1(v_1) = g_1(\delta v_1),\]
	\begin{equation*}\label{equ.conditiolinfty}
	\begin{split}
	(C_n): \quad dg_n(v_1,\ldots,v_n)&=
	(-1)^{1-n}\sum_{\sigma\in S(1,n-1)}\chi(\sigma)g_{n}(\delta v_{\sigma(1)},v_{\sigma(2)},\ldots,v_{\sigma(n)})\\
	&\quad +(-1)^{2-n}\sum_{\sigma\in S(2,n-2)}\chi(\sigma)g_{n-1}([v_{\sigma(1)},v_{\sigma(2)}],v_{\sigma(3)},\ldots,v_{\sigma(n)}).\end{split}\end{equation*}
\end{definition}

	\begin{remark}\label{rem.primacomp}
		Notice that condition $C_1$ entails that the linear component $g_1$ induces a map in cohomology $g_1 \colon H^*(V) \to H^*(M)$. It is clear that the cohomology $H^*(M)$ of an abelian DG-Lie algebra $M$ is an abelian graded Lie algebra. Condition $C_2$ can be written as
		\[ g_1([v_1, v_2]) = dg_2(v_1, v_2) + g_2 (\delta v_1, v_2) + (-1)^{\overline{v_1}}g_2 (v_1, \delta v_2),\]
		which implies that the map induced by $g_1$ in cohomology is a morphism of graded Lie algebras.
			\end{remark}

	Recall that since the functor $\Tot$ sends semicosimplicial DG-Lie algebras to DG-Lie algebras, the complex $\Tot(\sU, \LL)$ is a DG-Lie algebra. The complex of $\Oh_X$-modules $\Omega_X^{\leq 1} = \Oh_X \xrightarrow{\dr} \Omega_X^1$ can be considered as a sheaf of abelian DG-Lie algebras, and hence it gives rise to a semicosimplicial abelian DG-Lie algebra; therefore the complex $\Tot(\mathcal{U}, \Omega_X^{\leq 1} [2])$ is an abelian DG-Lie algebra.

	\begin{theorem}\label{teo.linfinitoalg}
	Let $(\mathcal{A}, \rho)$ be a transitive DG-Lie algebroid over a smooth separated scheme $X$ of finite type over a field $\K$ of characteristic zero. Let $\LL = \Ker\rho$ be a finite complex of locally free sheaves and let $\la-,-\ra \colon \Tot(\mathcal{U}, \Omega^i_X[-i] \otimes \LL) \times \Tot(\mathcal{U}, \Omega^j_X[-j] \otimes \LL) \to \Tot(\mathcal{U}, \Omega^{i+j}_X[-i-j] )$, $i,j \geq 0$, be a cyclic form which is $d_{\Tot}$-closed.
	For every simplicial lifting of the identity $D \in \Tot(\mathcal{U}, \Omega_X^1[-1] \otimes \mathcal{A})$ there exists a $L_\infty$ morphism between DG-Lie algebras on the field $\K$
	\[ f \colon \Tot(\mathcal{U}, \LL) \rightsquigarrow\Tot(\mathcal{U}, \Omega_X^{\leq 1} [2]) \]
	with components
	\begin{align*}
	f_1 (x) &= \langle u, x \rangle,\\
	f_2 (x,y) &= \frac{1}{2} \big( \la \nabla(x), y \ra - (-1)^{\overline{x}\ \overline{y}} \la \nabla (y), x \ra \big),\\
	f_3(x,y,z)  &=-\frac{1}{2} \la x, [y,z] \ra,\\
	f_n &= 0 \ \forall n \geq 4,
	\end{align*}
	where $\na= [D,-] \colon \Tot(\mathcal{U}, \LL) \to \Tot(\mathcal{U}, \Omega_X^1[-1] \otimes \LL)$ denotes the connection associated to the simplicial lifting of the identity $D $, and $u=d_{\Tot} D$ its extension cocycle. 
\end{theorem}

\begin{proof}
	The strategy of the proof it to check that the conditions $C_i$ of Definition~\ref{def.morfismo} hold for $n=1,2,3,4$. In fact, since $f_n=0$ for  $n\geq 4$, the conditions are automatically satisfied for $n \geq 5$.
	
	\begin{sloppypar}
		Denote by $d_{\Tot}$ the differential on $\Tot(\mathcal{U}, \LL)$, and by $d_{\Tot} + \dr$ the differential on $\Tot(\mathcal{U}, \Omega_X^{\leq 1} [2])$.
	Condition $C_1$ requires that $$f_1 (d_{\Tot}x) =(d_{\Tot} + \dr)f_1 (x);$$ notice however that since $u$ belongs to $\Tot(\mathcal{U}, \Omega_X^1[-1] \otimes \LL)$, $(d_{\Tot} + \dr)f_1 = d_{\Tot} f_1$ in $\Tot(\mathcal{U}, \Omega_X^{\leq 1} [2])$. Then by the $d_{\Tot}$-closure of the cyclic form and by the fact that $u$ is  closed:
	\end{sloppypar}
	\[ f_1 (d_{\Tot}x) = \la u, d_{\Tot}x \ra = (-1)^{\overline{u}}d_{\Tot}\la u,x \ra - (-1)^{\overline{u}}\la d_{\Tot}u, x \ra =d_{\Tot}\la u,x \ra =  d_{\Tot} f_1(x).\]
	
	For $n=2$ the condition is 
	\[ (C_2): \quad f_2 (d_{\Tot}x, y) + (-1)^{\overline{x}} f_2 (x, d_{\Tot}y) = f_1 ([x,y])-(d_{\Tot} + \dr)f_2(x,y).\]
	By definition of $f_2$ we have again that $(d_{\Tot} + \dr)f_2 = d_{\Tot} f_2$, and then, using Lemma~\ref{lem.diff e nabla},
	\begin{align*}
	&f_2 (d_{\Tot}x, y) + (-1)^{\overline{x}} f_2 (x, d_{\Tot}y) =\\
	& \frac{1}{2} \big( \la \nabla(d_{\Tot}x) , y \ra - (-1)^{(\overline{x} + 1)\overline{y}} \la \na(y) , d_{\Tot}x \ra + (-1)^{\overline{x}} \la \na (x), d_{\Tot}y \ra - (-1)^{\overline{x}\ \overline{y}} \la \na (d_{\Tot}y), x \ra \big)=\\
	& \frac{1}{2} \big( - \la d_{\Tot}\nabla(x) , y \ra  + \la [u,x], y \ra - (-1)^{(\overline{x} + 1)\overline{y}} \la \na(y) , d_{\Tot}x \ra + (-1)^{\overline{x}} \la \na (x), d_{\Tot}y \ra +\\
	&\qquad+ (-1)^{\overline{x}\ \overline{y}} \la d_{\Tot}\na (y), x \ra -(-1)^{\overline{x}\ \overline{y}}  \la [u,y], x \ra\big)= \\
	&\frac{1}{2} \big( \la u,[x,y] \ra - (-1)^{\overline{x}\ \overline{y}} \la u, [y, x] \ra -d_{\Tot} \la \na(x), y \ra + (-1)^{\overline{x}\ \overline{y}} d_{\Tot} \la \na (y), x \ra \big)  =\\
	& \la u, [x,y] \ra -\frac{1}{2} d_{\Tot} (\la \na (x), y \ra - (-1)^{\overline{x}\ \overline{y}} \la \na(y), x \ra ) = f_1 ([x,y])- d_{\Tot}f_2(x,y).
	\end{align*}
	
	Condition $C_3$ is the following
	\begin{align*}
	(d_{\Tot} +\dr) f_3 (x,y,z) &= f_3 (d_{\Tot}x,y,z) - (-1)^{\overline{x}\ \overline{y}} f_3 (d_{\Tot}y, x, z) + (-1)^{\overline{z}(\overline{x} + \overline{y})} f_3 (d_{\Tot}z, x, y) +\\
	&- f_2 ([x,y], z) 
	+ (-1)^{\overline{y}\ \overline{z}} f_2 ([x, z], y)  -(-1)^{\overline{x} (\overline{y} + \overline{z})} f_2 ([y, z], x),
	\end{align*}
	and we begin by noting that by the $d_{\Tot}$-closure
	\begin{align*}
	&f_3 (d_{\Tot}x,y,z) - (-1)^{\overline{x}\ \overline{y}} f_3 (d_{\Tot}y, x, z) + (-1)^{\overline{z}(\overline{x} + \overline{y})} f_3 (d_{\Tot}z, x, y)=\\
	& -\frac{1}{2} \big( \la d_{\Tot}x, [y,z] \ra - (-1)^{\overline{x}\ \overline{y}}  \la d_{\Tot}y, [x, z] \ra + (-1)^{\overline{z}(\overline{x} + \overline{y})} \la d_{\Tot}z, [x,y] \ra  \big)=\\
	& -\frac{1}{2} \big( \la d_{\Tot}x, [y,z] \ra - (-1)^{\overline{x}\ \overline{y}}  \la [d_{\Tot}y, x], z \ra + (-1)^{\overline{x} + \overline{y}} \la  [x,y] , d_{\Tot}z \ra  \big)=\\
	& -\frac{1}{2} \big( \la d_{\Tot}x, [y,z] \ra + (-1)^{\overline{x}}  \la [x, d_{\Tot}y], z \ra + (-1)^{\overline{x} + \overline{y}} \la  x,[y , d_{\Tot}z] \ra  \big)=\\
	& -\frac{1}{2} \big( \la d_{\Tot}x, [y,z] \ra + (-1)^{\overline{x}} \la x, d_{\Tot}[y,z] \ra \big) =-\frac{1}{2} d_{\Tot} \la x, [y,z] \ra = d_{\Tot} f_3(x,y,z).
	\end{align*}
	On the other hand, by Corollary~\ref{cor.nabladeRham}
	\begin{align*}
	&- f_2 ([x,y], z) + (-1)^{\overline{y}\ \overline{z}} f_2 ([x, z], y)  -(-1)^{\overline{x} (\overline{y} + \overline{z})} f_2 ([y, z], x)= \\
	& -\frac{1}{2} \big( \la \na ([x,y]), z \ra - (-1)^{\overline{z}(\overline{x} + \overline{y})} \la \na (z), [x,y] \ra - (-1)^{\overline{y}\ \overline{z}} \la \na ([x, z]), y \ra + (-1)^{\overline{x}\ \overline{y}}  \la \na (y), [x, z] \ra +\\
	&\qquad + (-1)^{\overline{x} (\overline{y} + \overline{z})} \la \na ([y,z]), x \ra - \la \na(x), [y,z] \ra \big) = \\
	& -\frac{1}{2} \big( \la  [\na(x),y], z \ra +(-1)^{\overline{x}} \la  [x,\na(y)], z \ra - (-1)^{\overline{z}(\overline{x} + \overline{y})} \la \na (z), [x,y] \ra - (-1)^{\overline{y}\ \overline{z}} \la [\na(x), z], y \ra +\\
	& \qquad- (-1)^{\overline{y}\ \overline{z}+\overline{x}} \la [x, \na (z)], y \ra+ (-1)^{\overline{x}\ \overline{y}}  \la \na (y), [x, z] \ra + (-1)^{\overline{x} (\overline{y} + \overline{z})} \la [\na(y),z], x \ra+ \\
	&\qquad+ (-1)^{\overline{x} (\overline{y} + \overline{z}) + \overline{y}} \la [y,\na(z)], x \ra- \la \na(x), [y,z] \ra \big) = \\ 
	& -\frac{1}{2} \big( \la \na(x), [y,z] \ra + (-1)^{\overline{x}} \la x, \na ([y,z]) \ra \big) = -\frac{1}{2} \dr \la x, [y,z] \ra = \dr f_3 (x,y,z).
	\end{align*}

	Lastly, condition $C_4$ is
	\begin{align*}
	&f_3 ([a_1, a_2], a_3, a_4) - (-1)^{\overline{a_2}\ \overline{a_3}} f_3 ([a_1, a_3], a_2, a_4 ) + (-1)^{\overline{a_4}(\overline{a_2} + \overline{a_3})} f_3 ([a_1, a_4], a_2, a_3)\\
	&\qquad + (-1)^{\overline{a_1}(\overline{a_2} + \overline{a_3})} f_3 ([a_2, a_3], a_1, a_4) - (-1)^{\overline{a_3}\ \overline{a_4} + \overline{a_1}\ \overline{a_2} + \overline{a_1}\ \overline{a_4}} f_3 ([a_2, a_4], a_1, a_3)  \\
	&\qquad +(-1)^{(\overline{a_1} + \overline{a_2})(\overline{a_3}+ \overline{a_4})}f_3 ([a_3, a_4], a_1, a_2)=0.
	\end{align*}
	By the graded Jacobi identity we have that 
	\begin{align*}
	&-\frac{1}{2}\big(\langle [a_1, a_2], [a_3, a_4] \rangle - (-1)^{\overline{a_2}\ \overline{a_3}} \langle [a_1, a_3], [a_2, a_4 ] \rangle  + (-1)^{\overline{a_4}(\overline{a_2} + \overline{a_3})} \langle [a_1, a_4], [a_2, a_3] \rangle +\\
	&\qquad+ (-1)^{\overline{a_1}(\overline{a_2} + \overline{a_3})} \langle [a_2, a_3], [a_1, a_4]\rangle  - (-1)^{\overline{a_3}\ \overline{a_4} + \overline{a_1}\ \overline{a_2} + \overline{a_1}\ \overline{a_4}} \langle  [a_2, a_4], [a_1, a_3] \rangle   \\
	&\qquad +(-1)^{(\overline{a_1} + \overline{a_2})(\overline{a_3}+ \overline{a_4})}\langle [a_3, a_4],[a_1, a_2]\rangle \big)= \\
	&-(\langle [a_1, a_2], [a_3, a_4] \rangle - (-1)^{\overline{a_2}\ \overline{a_3}} \langle [a_1, a_3], [a_2, a_4 ] \rangle + (-1)^{\overline{a_4}(\overline{a_2} + \overline{a_3})} \langle [a_1, a_4], [a_2, a_3] \rangle)=\\
	& -( \langle a_1, [a_2, [a_3, a_4]] \rangle - (-1)^{\overline{a_2}\ \overline{a_3}} \langle a_1, [a_3, [a_2, a_4 ]] \rangle + (-1)^{\overline{a_4}(\overline{a_2} + \overline{a_3})} \langle a_1, [a_4, [a_2, a_3]] \rangle )=\\
	&  -\langle a_1, [a_2, [a_3, a_4]] -(-1)^{\overline{a_2}\ \overline{a_3}} [a_3, [a_2, a_4 ]] -  [[a_2, a_3], a_4]\rangle =0\,.
	\end{align*}
\end{proof}

\medspace

We can now state the results of  \cite{linfsemireg} for a coherent sheaf admitting a finite locally free resolution on a smooth separated scheme $X$ of finite type on a field $\K$ of characteristic zero.

\begin{remark}
	It is not very restrictive to require that a coherent sheaf on $X$
	has a finite locally free resolution: in fact, by \cite[III, Exercises 6.8, 6.9]{Hart} every 
	 coherent sheaf on a smooth, Noetherian, integral, separated scheme admits a finite locally free resolution.
\end{remark}

Let $(\sE, d_{\sE})$ be a finite complex of locally free sheaves. 
Consider the DG-Lie algebroid of derivations of pairs $\sD^*(X, \sE)$ of Example~\ref{ex.coppie}, \cite{DMcoppie}, and the short exact sequence
\begin{equation*}
	\begin{tikzcd}
	0 \arrow[r] & {\HOM^*_{\Oh_X}(\sE,\sE)} \arrow[r] & {\sD^*(X, \sE)} \arrow[r, "\alpha"] & \Theta_X \arrow[r] & 0;
	\end{tikzcd}
\end{equation*}
it was noted in Example~\ref{ex.coppie} that by tensoring with $\Omega_X^1[-1]$ one obtains 
\[
0\to  \HOM^*_{\Oh_X}(\sE,\Omega^1_X[-1]\otimes \sE)\xrightarrow{g\mapsto (0,g)} \sJ^*_{\Omega^1}\xrightarrow{(\beta, g)\mapsto\beta} 
\DER_{\K}(\Oh_X,\Omega^1_X[-1])\to 0\,,
\]
Fixing an affine open cover $\sU$ of $X$ and applying the $\Tot$ functor, we get the short exact sequence
\[ 0\to  \Tot(\sU,\HOM_{\Oh_X}^*(\sE,\Omega^1_X [-1]\otimes \sE))\xrightarrow{\quad} 
\Tot(\sU,\mathcal{J}_{\Omega^1}^*)\xrightarrow{\quad } 
\Tot(\sU,\DER_{\K}(\Oh_X,\Omega^1_X[-1]))\to 0\,.\]
We have already remarked that a lifting of the identity in $\sJ^*_{\Omega^1}$ is equivalent to a global algebraic connection on every component $\sE^i$; hence a lifting to $\Tot(\sU,\mathcal{J}_{\Omega}^*)$ of the universal derivation $\dr \colon \Oh_X \to \Omega_X^1[-1] $ in $\Tot(\sU,\DER_{\K}(\Oh_X,\Omega^1_X[-1]))$  can be termed a \textbf{simplicial connection} on the complex of locally free sheaves $\sE$.
As seen in Example~\ref{ex.traccia}, a natural cyclic form to consider is the one induced by
\[\HOM^*_{\Oh_X}(\sE,\sE) \times\HOM^*_{\Oh_X}(\sE,\sE) \to \Oh_X, \quad (a, b) \mapsto -\Tr ( ab),\] where $\Tr \colon \HOM^*_{\Oh_X} (\sE,\sE) \to \Oh_X$ is the usual trace map. 
Then the $L_\infty$ morphism of Theorem~\ref{teo.linfinitoalg} yields :

\begin{corollary}\label{cor.linfBF} 
	Let $\sE$ be a finite complex of locally free sheaves on a smooth separated scheme $X$ of finite type over a field $\K$ of characteristic zero. For every simplicial connection $D \in \Tot(\sU, \sJ^*_{\Omega^1})$ 
	there exists an
	$L_\infty$  morphism between DG-Lie algebras on the field $\K$
	\[g\colon \Tot(\mathcal{U},\HOM^*_{\Oh_X} (\sE,\sE)) \rightsquigarrow \Tot(\mathcal{U}, \Omega_X^{\leq 1} [2])\] 
	with components 
	\begin{align*}
	g_1 (x) &= -\Tr (ux),\\
	g_2 (x,y) &= -\frac{1}{2} \Tr (\nabla(x) y - (-1)^{\overline{x}\ \overline{y}} \nabla (y) x ),\\
	g_3(x,y,z)  &=\frac{1}{2}\Tr (  x, [y,z] ),\\
	g_n &= 0 \ \forall n \geq 4.
	\end{align*}
\end{corollary}

Hence the applications to deformation theory of \cite{linfsemireg}, stated in the context of complex manifolds, are also valid in the algebraic context as announced.
Let $\sF$ be a coherent sheaf on $X$ admitting a finite locally free resolution, and denote by  
\[ \sigma=\sum_{q\ge 0}\sigma_q\colon \Ext^2_X(\sF,\sF)\to \prod_{q\ge 0}H^{q+2}(X,\Omega_X^q),\qquad 
\sigma(c)=\Tr(\exp(-\At(\sF))\circ c)\,,\] 
the Buchweitz-Flenner semiregularity map of \cite{BF}. Above, $\At(\sF) \in \Ext_X^1(\sF, \sF \otimes \Omega_X^1)$ denotes the Atiyah class of $\sF$, the exponential of its opposite 
\[ \exp(-\At(\sF))\in \prod_{q\ge 0}\Ext^q_X(\sF,\sF\otimes \Omega^q_X)\,\]
is obtained via the Yoneda pairing
\[ \Ext^i_X(\sF,\sF\otimes \Omega^i_X)\times \Ext^j_X(\sF,\sF\otimes \Omega^j_X)\to 
\Ext^{i+j}_X(\sF,\sF\otimes \Omega^{i+j}_X),\qquad (a,b)\mapsto a\circ b,\]
and $\Tr$ denotes the trace maps
\[ \Tr\colon \Ext_X^i(\sF,\sF\otimes \Omega^j_X)\to H^i(X,\Omega^j_X),\qquad i,j\ge 0\,.\]
For every $q \geq 0$ one can consider the composition 
\[  \tau_q\colon \Ext_X^2(\sF,\sF)\xrightarrow{\sigma_q}H^{q+2}(X,\Omega_X^q)=H^{2}(X,\Omega_X^q[q])\xrightarrow{i_q} \mathbb{H}^{2}(X,\Omega_X^{\le q}[2q]),\]
where $\Omega^{\le q}_X=(\oplus_{i=0}^q\Omega_X^i[-i],\dr)$ is the truncated  de Rham complex and $i_q$ is induced by the inclusion of complexes $\Omega_X^{q}[q]\subset \Omega_X^{\le q}[2q]$. The map $\tau_q$ is the { $q$-component of the modified Buchweitz-Flenner semiregularity} map. 
It is convenient to also consider the maps
\begin{align*}
&\sigma_q\colon \Ext^*_X(\sF,\sF)\to H^{*}(X,\Omega_X^q[q])\\
&\tau_q\colon \Ext_X^*(\sF,\sF) \xrightarrow{\sigma_q}H^{*}(X,\Omega_X^q[q])\xrightarrow{i_q} \mathbb{H}^{*}(X,\Omega_X^{\le q}[2q])
\end{align*}
defined by the same formulas.

\begin{corollary}\label{cor.fasci1} 
	Let $\sF$ be a coherent sheaf admitting a finite locally free resolution $\sE$ on a smooth separated scheme $X$ of finite type over a field $\K$ of characteristic zero. Then every simplicial connection on the resolution $\sE$ gives a lifting of the map
	\[  \tau_1\colon \Ext_X^*(\sF,\sF)\to \mathbb{H}^{*}(X,\Omega_X^{\le 1}[2])\]
	to an $L_{\infty}$ morphism 
	\[ g\colon \Tot(\mathcal{U},\HOM^*_{\Oh_X} (\sE,\sE)) \rightsquigarrow \Tot(\mathcal{U}, \Omega_X^{\leq 1} [2]).\]
\end{corollary}

Recall that to a DG-Lie algebra $M$ over a field $\K$ of characteristic zero we can associate a functor $\Def_M \colon \Art_\K \to \Set$, the functor of Maurer-Cartan solutions modulo gauge action (for more details see e.g. \cite{sheaves,  semireg,DMcoppie, ManRendiconti,LMDT}). It is well known that $H^2(M)$ is an obstruction space for the deformation functor $\Def_M$. Recall also that an $L_\infty$ morphism between DG-Lie algebras $g \colon V \rightsquigarrow M$ gives a morphism of deformation functors $g \colon \Def_V \to \Def_M$ such that the map induced in cohomology commutes with obstruction maps. If the DG-Lie algebra $M$ has trivial bracket, every obstruction in $\Def_M$ is trivial, and therefore every obstruction in $\Def_V$ belongs to the kernel of the map $g \colon H^2(V) \to H^2(M)$.

\begin{corollary}\label{cor.fasci2} 
	Let $\sF$ be a coherent sheaf admitting a finite locally free resolution on a smooth separated scheme $X$ of finite type over a field $\K$ of characteristic zero. Then every obstruction to the deformations of $\sF$ belongs to the kernel of the map
	\[  \tau_1\colon \Ext_X^2(\sF,\sF)\to \mathbb{H}^{2}(X,\Omega_X^{\le 1}[2]).\]
	{If the Hodge to de Rham spectral sequence of $X$ degenerates at $E_1$, then  
	every obstruction to the deformations of $\sF$ belongs to the kernel of the map
	\[  \sigma_1\colon \Ext_X^2(\sF,\sF)\to {H}^{3}(X,\Omega_X^{1}),\qquad \sigma_1(a)=-\Tr(\At(\sF)\circ a).\]}
\end{corollary}

\begin{proof}
	
	If $\sE$ is a finite locally free resolution of $\sF$, the DG-Lie algebra $\Tot(\mathcal{U},\HOM^*_{\Oh_X} (\sE,\sE))$ controls the deformations of $\sF$, see e.g. \cite{sheaves}. According to Corollary~\ref{cor.fasci1}, the map
	\[  \tau_1\colon \Ext_X^2(\sF,\sF)\to \mathbb{H}^{2}(X,\Omega_X^{\le 1}[2])\]
	lifts to an $L_\infty$ morphism \[ g\colon \Tot(\mathcal{U},\HOM^*_{\Oh_X} (\sE,\sE)) \rightsquigarrow \Tot(\mathcal{U}, \Omega_X^{\leq 1} [2]),\]
	whose linear component $g_1$ commutes with obstruction maps  of the associated deformation functors. By construction the DG-Lie algebra $\Tot(\mathcal{U}, \Omega_X^{\leq 1} [2])$ is abelian and therefore every obstruction of the associated deformation functor is trivial.

	If the Hodge to de Rham spectral sequence of $X$ degenerates at $E_1$ then the inclusion of complexes $ \Tot(\mathcal{U}, \Omega^1_X[1])\to \Tot(\mathcal{U}, \Omega_X^{\leq 1} [2])$ is injective in cohomology, so that  $H^3(X,\Omega_X^1) \hookrightarrow \mathbb{H}^2(X,\Omega_X^{\leq 1}[2])$ and the maps $\sigma$ and $\tau$ have the same kernel.
\end{proof}

\begin{remark}
		In the setting of Theorem~\ref{teo.linfinitoalg}, if the cyclic form is induced by a DG-Lie algebroid representation $\theta \colon \mathcal{A} \to \sD^*(X, \sE)$, the $L_\infty$ morphism can be obtained up to a sign from the the $L_\infty$ morphism of Corollary~\ref{cor.linfBF} as follows. 
		Let $D \in \Tot(\sU, \Omega_X^1[-1] \otimes \mathcal{A})$ denote a simplicial lifting of the identity, and denote by $\Id \otimes \theta \colon \Tot(\sU, \Omega_X^1[-1] \otimes \mathcal{A}) \to \Tot(\sU, \Omega_X^1[-1] \otimes \sD^*(X, \sE)) \cong \Tot(\sU, \sJ^*_{\Omega})$ the induced map on the totalisation. Denoting as usual by $\alpha$ the anchor map of the transitive DG-Lie algebroid $\sD^*(X,\sE)$, it is clear that $(\Id \otimes \theta) (D)$ is a simplicial lifting of the identity in $\Tot(\sU, \Omega_X^1[-1] \otimes \sD^*(X, \sE))$:
		\[ (\Id \otimes \alpha)(\Id \otimes \theta) (D)= \Id \otimes (\alpha \circ \theta)(D)= (\Id \otimes \rho )(D) = \Id_{\Omega^1} \in \Tot(\sU, \Omega_X^1[-1] \otimes \Theta_X).\]
		
		Let $u = d_{\Tot}D \in \Tot(\sU, \Omega_X^1[-1] \otimes \LL)$ denote the extension  cocycle associated to $D$, then 
		\[ (\Id \otimes \theta )(u)= (\Id \otimes \theta)(d_{\Tot} D)= d_{\Tot} (\Id \otimes \theta)(D).\]
		Therefore the $L_\infty$ morphism $f \colon \Tot(\sU, \LL) \rightsquigarrow \Tot(\sU, \Omega_X^{\leq1}[2])$ associated to $D$ and to $\la-,-\ra_\theta $ is the composition of the DG-Lie algebra morphism
		\[\theta \colon \Tot(\sU, \LL)  \to \Tot(\sU, \HOM^*_{\Oh_X}(\sE,\sE))\]
		and of the $L_\infty$ morphism 
		\[- g\colon \Tot(\mathcal{U},\HOM^*_{\Oh_X} (\sE,\sE)) \rightsquigarrow \Tot(\mathcal{U}, \Omega_X^{\leq 1} [2])\] associated to the simplicial lifting of the identity $(\Id \otimes \theta)(D)$ and to the cyclic form $(a,b) \mapsto \Tr (ab)$.
\end{remark}

	\section{The $L_\infty$ morphism for the Atiyah Lie algebroid of a principal bundle}\label{sec.fibratiprinc}
	
	Since Lie algebroids arise naturally in connection with principal bundles, we give an application of the $L_\infty$ morphism constructed in Theorem~\ref{teo.linfinitoalg} to the deformation theory of principal bundles.
	
	Let $X$ be a smooth separated scheme of finite type over an algebraically closed field $\K$ of characteristic zero,
	let $G$ be an affine algebraic group with Lie algebra $\g$, and let $P \to X$ be a principal $G$-bundle on $X$. By $G$-principal bundle we mean a $G$-fibration which is locally trivial for the Zariski topology, see e.g. \cite{Sor}.
	We begin by finding a DG-Lie algebra that controls the deformations of $P$, using an argument similar to those in \cite{Bis, LMDT,Ueno} . Let $\Art_\K$ be the category of Artin local $\K$-algebras with residue field $\K$.	For any  $A$ in $\Art_\K$   denote  by $\m_A$ its maximal ideal and by $0$ the closed point in $\Spec A$.
	
	To every semicosimplicial Lie algebra $\mathfrak{h}$ over  $\K$ 
	\begin{center}
		\begin{tikzcd}
		\mathfrak{h}: & \mathfrak{h}_0 \arrow[r, "\delta_0", shift left] \arrow[r, "\delta_1"', shift right] & \mathfrak{h}_1 \arrow[r, "\delta_0", shift left=2] \arrow[r, "\delta_2"', shift right=2] \arrow[r, "\delta_1" description] & \mathfrak{h}_2 \arrow[r, shift left=3] \arrow[r, shift left] \arrow[r, shift right=3] \arrow[r, shift right] & \cdots
		\end{tikzcd}
	\end{center}
	there are associated two functors $Z^1_\mathfrak{h}, H^1_\mathfrak{h} \colon \Art_\K \to \Set$, which here are described in brief; for more details see \cite{semicos,LMDT}. The functor of non-abelian cocycles $Z^1_\mathfrak{h}$ is defined as 
	\[ Z^1_\mathfrak{h}(A) = \{ e^x \in \exp (\mathfrak{h}_1 \otimes \m_A) \ | \ e^{\delta_1(x)} = e^{\delta_2(x)}e^{\delta_0(x)}\}.\]
	For every $A \in \Art_\K$ there is a left action of $\exp (\mathfrak{h}_0\otimes \m_A) $ on $Z^1_\mathfrak{h} (A)$
	\[ \exp (\mathfrak{h}_0\otimes \m_A) \times Z^1_\mathfrak{h} (A) \to Z^1_\mathfrak{h}(A),\ \ (e^a, e^x)\mapsto e^{\delta_1(a)} e^x e^{-\delta_0(a)}.\]
	The functor $H^1_\mathfrak{h} \colon \Art_\K \to \Set$ is then defined as 
	\[ H^1_\mathfrak{h}(A) = \frac{Z^1_\mathfrak{h}(A)}{\exp (\mathfrak{h}_0\otimes \m_A)}.\]
	
	Consider the Thom-Whitney totalisation functor $ \Tot$ from semicosimplicial DG-vector spaces to
	DG-vector spaces 
	(Definition~\ref{def.tot}), and recall it takes semicosimplicial Lie algebras to DG-Lie algebras.  We then have the following result, see  \cite{semicos,Hin,LMDT}:
	\begin{proposition}\label{prop.iso-funtori}
		For every semicosimplicial Lie algebra $\mathfrak{h}$ there exists a natural isomorphism of functors $H_\mathfrak{h}^1 \cong \Def_{\Tot(\mathfrak{h})}$ .
	\end{proposition}

	\begin{definition}\cite{Bis, Don}
		An infinitesimal deformation of $P$ over $A \in \Art_\K$ is the data of a principal $G$-bundle $P_A \to X \times \Spec A$ and an isomorphism $\theta \colon i^* (P_A) \cong P$.
		\begin{center}
			\begin{tikzcd}
			P \arrow[d, "p"'] \arrow[r] & P_A \arrow[d, "p_A"] \\
			X \arrow[r, "i"]            & X\times \Spec A .    
			\end{tikzcd}
		\end{center}	
		Two deformations $(P_A, \theta)$ and $(P_A, \theta')$ are isomorphic if there exists an isomorphism of principal $G$-bundles $\lambda \colon P_A \to P'_A$ such that $\theta = \theta' \circ i^* (\lambda)$.
		
	\end{definition}
	This defines a functor $\Def_P \colon \Art_\K \to \Set$ such that $\Def_P(A)$ is the set of isomorphism classes of deformations of $P$ over $A \in \Art_\K$. For every $A \in \Art_\K$, the set $\Def_P(A)$ contains the trivial deformation $P \times \Spec A \to X \times \Spec A$.

			If $M$ is a DG-Lie algebra such that $\Def_P \cong \Def_M$, where $\Def_M$ is the functor of Maurer-Cartan solutions modulo gauge action, one says that $M$ controls the deformations of $P$. 
		
	Fix an open cover $\sU = \{ U_i\}$ of $X$  such that $P$ is trivial on every $U_i$, and let $\{g_{ij} \colon U_{ij}  \to G\}$ denote the transition functions for $P$.
	Let $\g$ be the Lie algebra of $G$.
	\begin{itemize}
		\item 	Let $\ad P = P \times^G \g$ denote the adjoint bundle of $P$, with transition functions $\{\Ad_{g_{ij}}\} $, and let $\mathscr{ad}(P)$ denote the sheaf of sections of the vector bundle $\ad P$.
		\item The group $G$ acts on itself by conjugation; denote by $\Ad P = P \times^G G$ the associated bundle corresponding to this action. Recall that  $ \Gamma (X,\Ad(P)) \cong \Gau(P)$, where  $\Gau(P)$ is the group of bundle automorphisms of $P$. 
	\end{itemize}
	
	There is a one to one correspondence between first order deformations of $P$, i.e., deformations over $\K[t]/(t^2) \in \Art_\K$, and $H^1(X, \Sad)$, see e.g. \cite{Don,Ueno}. This implies that on every affine open set the deformations of $P$ are trivial.
	
	\begin{lemma}\label{lem.sezioni}
		Let $P \times^G (\g \otimes \m_A)$ be the associated bundle induced by the action  $\Ad \otimes \Id \colon G \times \g\otimes \m_A \to \g \otimes \m_A$. Then there is an isomorphism
		\[   \Gamma (P \times^G (\g \otimes \m_A)) \cong \Gamma (\Sad) \otimes \m_A.\]
	\end{lemma}
	
	\begin{proof}	
		A section of $P \times^G (\g \otimes \m_A)$ is the data of 
		\[ \{ \omega_i \colon U_i \to \g \otimes \m_A\ | \ \omega_i(p) = (\Ad_{g_{ij}(p)} \otimes \Id) \omega_j(p) \ \ \forall p \in U_{ij} \}.\]
		Let $t_1, \cdots, t_n$ be a basis of the finite dimensional vector space $\m_A$, then for every $p \in U_i$ one can write
		$\omega_i(p) = \sum_k  h_{i,k}(p) \otimes t_k$. Since  the action of $G$ on $\g \otimes \m_A$ is defined as 
		\[ g \cdot (x \otimes t) = \Ad_g (x) \otimes t,\] the maps $h_{i,k}$ are such that $h_{i,k}(p) = \Ad_{g_{ij}(p)} h_{j,k} (p)$ for every $p \in U_{ij}.$ 
		
		An element of $\Gamma( \Sad ) \otimes \m_A$ is a finite sum $\sum_k \eta_k \otimes t_k$, with $\eta_k $ sections of $\ad P$, so that each $\eta_k$ is the data of
		\[ \{ \eta_{k,i} \colon U_i \to \g \ | \ \eta_{k,i}(p) = \Ad_{g_{ij}(p)} \eta_{k,j} (p) \ \ \forall p \in U_{ij}  \}.\]
		Then, setting $(\eta_{k,i} \otimes t_k )(p)= \eta_{k,i} (p) \otimes t_k  $ for every $p \in U_i$, the data 
		$\{ \eta_{k,i} \otimes t_k \colon U_i \to \g \otimes \m_A  \}$ is exactly a section of $P \times^G (\g \otimes \m_A)$.

	\end{proof}
	
	\begin{lemma}\label{lem.automorfismi}
		For every $A \in \Art_\K$ there is an isomorphism of groups 
		\[ \exp (\Gamma( \Sad ) \otimes \m_A) \cong \{ \text{automorphisms of the trivial deformation } P \times \Spec A \}.\]
	\end{lemma}
	
	\begin{proof}

		Denote by $G^0(A) $ the group of morphisms $f \colon \Spec A \to G$ such that $f(0) = \Id_G$, and recall that there is an isomorphism of groups $\exp (\g \otimes \m_A) \cong G^0(A)$ (see e.g. \cite[Section 10]{Sim}). The group structure on $G^0(A)$ is induced by the group structure on $G$, while $\exp (\g \otimes \m_A)$ is a group with the Baker-Campbell-Hausdorff product. 
		By Lemma~\ref{lem.sezioni}  $$\Gamma( \Sad ) \otimes \m_A \cong \Gamma (P \times^G (\g \otimes \m_A)),$$ so that we can work with $\exp (\Gamma (P \times^G (\g \otimes \m_A))).$
		Consider the associated bundle $P \times^G G^0(A)$, induced by the adjoint action of $G$ on $G^0(A)$; the isomorphism $\exp (\g \otimes \m_A) \cong G^0(A)$ induces an isomorphism
		$\exp (\Gamma (P \times^G (\g \otimes \m_A)))  \cong \Gamma (P \times^G G^0(A)) $.
		In fact, a section of $P \times^G (\g \otimes \m_A)$ is the data of 
		\[ \{\eta_i \colon U_i \to \g \otimes \m_A \ | \ \eta_i(p) = (\Ad_{g_{ij}(p)} \otimes \Id) \eta_j (p) \ \ \forall p \in U_{ij}\},\]
		and composing with the exponential $\exp \colon \g \otimes \m_A \to G^0(A)$ we obtain 
		\[ \{\exp \circ \eta_i \colon U_i \to G^0(A) \ | \ \exp \circ \eta_i(p) = {g_{ij}(p)}\exp \circ  \eta_j (p){g_{ij}(p)}^{-1} \ \ \forall p \in U_{ij}\}. \]
		
		Notice that this data is equivalent to
		\[\{ \lambda_i \colon U_i \times \Spec A \to G \ | \ \lambda_i(p,0) = \Id_G \ \  \forall p \in U_i,\ \ \lambda_i (p) = g_{ij}(p) \lambda_j (p) g_{ij}(p)^{-1} \ \ \forall p \in U_{ij} \},\]
		which is a section of the associated bundle $ \Ad (P \times \Spec A) = (P \times \Spec A) \times^G G$, where $G$ acts on itself by conjugation.
		
		For any $G$-principal bundle $Q$ the global sections of the associated bundle $\Ad(Q) = Q \times^G G$ correspond to bundle automorphisms of $Q$. 
		Therefore the $\{\lambda_i\}$ give an element $F \in \Gau(P \times \Spec A)$, and the condition $\lambda_i(p,0) = \Id_G$ for all $p \in U_i$ is equivalent to the fact that the automorphism $F$ induces the identity when restricted to $P$, so that $F$ is an automorphism of the trivial deformation.
	\end{proof}

	\begin{proposition}\label{prop.deformazioni}
		Let $\mathcal{U} = \{ U_i\}$ be an affine open cover of $X$ and let $\Sad (\mathcal{U})$ be the semicosimplicial Lie algebra of \v{C}ech cochains:
		\begin{center}
			\begin{tikzcd}
			\prod_i \Sad (U_i) \arrow[r, shift left] \arrow[r, shift right] & {\prod_{i,j}\Sad(U_{ij})} \arrow[r, shift left=2] \arrow[r, shift right=2] \arrow[r] & {\prod_{i,j,k}\Sad(U_{ijk})} \arrow[r, shift left=3] \arrow[r, shift left] \arrow[r, shift right=3] \arrow[r, shift right] & \cdots.
			\end{tikzcd}
		\end{center}
		There is a natural isomorphism of functors $H^1_{\Sad(\mathcal{U})} \to \Def_P $.
	\end{proposition}
	
	\begin{proof}
		Recall that all deformations of $P$ on an affine open set are trivial, as mentioned above.
		Fix $A \in \Art_\K$; by Lemma~\ref{lem.automorfismi} an element $f$ of $Z^1_{\Sad(\mathcal{U})} (A)$ is the data for every $U_{ij}$ of isomorphisms $f_{ij} \colon P|_{U_{ij}} \times \Spec A \to   P|_{U_{ij}} \times \Spec A$, which restrict to the identity  $P|_{U_{ij}}  \to P|_{U_{ij}} $ and such that $f_{ik} = f_{ij} f_{jk}$ for all $i,j,k$. 
		
		The last condition means that the $\{f_{ij}\}$ glue to obtain a principal $G$-bundle $P_A \to X \times \Spec A$ and isomorphisms $f_i \colon P_A|_{U_i \times \Spec A} \to P|_{U_i} \times \Spec A$ such that $f_{ij} = f_i f_j^{-1}$. Such isomorphisms coincide when restricted to 
		$\overline{f_i} \colon i^* (P_A|_{U_i \times \Spec A} ) \to P|_{U_i}$ and hence glue to an isomorphism of principal bundles $ i^*(P_A) \to P$. This means that an element of $Z^1_{\Sad(\mathcal{U})} (A)$ gives a locally trivial deformation of $P$ over $A \in \Art_\K$.
		
		An element of $\exp(\prod_i \Sad (U_i) \otimes \m_A)$ is again by Lemma~\ref{lem.automorfismi} the data, for every $U_i$, of automorphisms $\lambda_i \colon  P|_{U_{i}} \times \Spec A \to   P|_{U_{i}} \times \Spec A$ which restrict to the identity $P|_{U_{i}}  \to P|_{U_{i}} $.
		Two elements $f =\{f_{ij}\}, h= \{h_{ij}\}$ of $Z^1_{\Sad(\mathcal{U})} (A)$ are equivalent under the action of $\lambda \in \exp(\prod_i \Sad (U_i) \otimes \m_A)$ if and only if $ h_{ij} = \lambda_i f_{ij} \lambda_j^{-1}$ for all $i,j$. 
		\begin{center}
			\begin{tikzcd}
			P_A|_{U_i \times \Spec A} \arrow[r, "f_i"] \arrow[d, "\lambda"'] & P|_{U_i} \times \Spec A \arrow[d, "\lambda_i"] \\
			P_A'|_{U_i \times \Spec A} \arrow[r, "h_i"']                     & P|_{U_i} \times \Spec A                       
			\end{tikzcd}
		\end{center}
		This can be expressed as $h_i^{-1} \lambda_i f_i =h_j^{-1} \lambda_jf_j$, which means that the $\{\lambda_i\}$ glue to a bundle isomorphism $\lambda \colon P_A \to P'_A$, where $P_A$ is the deformation corresponding to $\{f_{ij}\}$, and $P'_A$ to $\{h_{ij}\}$. Since each $\lambda_i$ restricts to the identity on $P|_{U_i}$, $\lambda$ is an isomorphism of deformations. 
	\end{proof}

\begin{corollary}\label{cor.dgladef}
	If $\mathcal{U}=\{U_i\}$ is an affine open cover of $X$, there is an isomorphism 	$$\Def_P \cong \Def_{\Tot(\mathcal{U}, \Sad)},$$ i.e., 
	the DG-Lie algebra $\Tot(\mathcal{U}, \Sad  )$ controls the deformations of $P$.
\end{corollary}

\begin{proof}
	Consequence of  Propositions~\ref{prop.deformazioni} and \ref{prop.iso-funtori} .
\end{proof}

\medspace

	We now specialise the $L_\infty$ morphism of Section~\ref{sec.cyclic} to the Atiyah Lie algebroid of the principal $G$-bundle $P$. 
	
		A Lie algebroid is DG-Lie algebroid (Definition~\ref{def.DGLiealg}) concentrated in degree zero.
	Consider the \textbf{Atiyah Lie algebroid} of the principal bundle $P$ introduced in \cite{At},  which is a Lie algebroid structure on the sheaf $\sQ$ of sections of the vector bundle $Q= T_P/G$, the quotient of the tangent bundle of the total space $T_P$ by the canonical induced $G$-action.
	There is a canonical short exact sequence of locally free sheaves over $X$
	\begin{equation}\label{eq.atiyahLiealg}
	\begin{tikzcd}
	0 \arrow[r] & \Sad  \arrow[r] & \sQ \arrow[r, "\rho"] & \Theta_X \arrow[r] & 0,
	\end{tikzcd}
	\end{equation}
	where $\Sad$ denotes the sheaf of sections of the adjoint bundle $\ad P = P \times^G \g$  and $\rho \colon \sQ \to \Theta_X$ is the anchor map. The vector bundle $Q$ is the bundle of invariant tangent vector fields on $P$, and the Lie bracket on $\sQ$ is induced by the Lie bracket of vector fields. 
	
	\begin{definition}\cite{At}\label{def.connAT}
		A {connection} on the principal bundle $P \to X$ is a splitting of the exact sequence in \eqref{eq.atiyahLiealg}.
	\end{definition}
	It is clear that a connection on $P$ does not always exist.
	
	Let $\Omega_X^1$ denote   the cotangent sheaf, and $\Omega_X^1[-1]$ the cotangent sheaf considered as a trivial complex of sheaves concentrated in degree one.
	As in Section~\ref{sec.DGLiealg}, one can tensor the short exact sequence  \eqref{eq.atiyahLiealg}
 with $\Omega^1_X[-1]$ to obtain a short exact sequence of complexes of sheaves
	\begin{equation*}
	\begin{tikzcd}
	0 \arrow[r] &   \Omega_X^1[-1]\otimes \Sad \arrow[r] &   \Omega_X^1[-1]\otimes \sQ\arrow[r, "\Id\otimes \rho"] & \Omega_X^1[-1]\otimes  \Theta_X \arrow[r] & 0.
	\end{tikzcd}
	\end{equation*}

	Fix an affine open cover $\sU = \{ U_i\}$ of $X$;
	as in Section~\ref{sec.DGLiealg}  the short exact sequence above
	induces a short exact sequence of DG-vector spaces
	\begin{center}
		\begin{tikzcd}[column sep=small]
		0 \arrow[r] & \Tot(\mathcal{U},  \Omega_X^1[-1]\otimes\Sad ) \arrow[r] & \Tot(\mathcal{U},\Omega_X^1[-1]\otimes  \sQ ) \arrow[r, "\Id\otimes \rho"] & \Tot(\mathcal{U}, \Omega_X^1[-1]\otimes \Theta_X) \arrow[r] & 0,
		\end{tikzcd}
	\end{center}
	and we denote by $d_{\Tot}$ the differentials of the above complexes.
	
	 It is easily seen that a lifting  of the identity $\Id_{\Omega^1}\in \Gamma(X,\Omega_X^1[-1] \otimes \Theta_X)$ to $D \in \Gamma(X,\Omega_X^1[-1]\otimes \sQ)$ is equivalent to
	a splitting of the exact sequence in \eqref{eq.atiyahLiealg}.
	Hence in the case of a principal bundle $P$,  a lifting of the identity can be identified with a connection on $P$. Therefore we call  a preimage of $\Id_{\Omega^1}$ in $  \Omega_X^1[-1]\otimes \sQ$ a germ of a connection on $P$, and we use the following terminology:
	\begin{definition}
		A \textbf{simplicial {connection}} on the principal bundle $P$ is a lifting $D$ in \\$\Tot(\mathcal{U},   \Omega_X^1[-1]\otimes \sQ)$ of the identity $ \Id_{\Omega^1} $ in  $\Tot(\mathcal{U}, \Omega_X^1[-1]\otimes \Theta_X)$.
	\end{definition}
	
	\begin{definition}
		The \textbf{{Atiyah} cocycle} of $P$ is $u= d_{\Tot}D \in \Tot (\mathcal{U}, \Omega_X^1[-1]\otimes \Sad )$.
	\end{definition}

It is natural to use the name Atiyah cocycle instead of extension cocycle of Definition~\ref{def.Atiyah-cocycle}, because its cohomology class is equal to the extension class of the short exact sequence in \eqref{eq.atiyahLiealg}, hence it vanishes if and only if there exists a connection on $P$.

	As in Definition~\ref{def.connessione}, given a simplicial {connection} $D \in  \Tot(\mathcal{U},  \Omega_X^1[-1]\otimes\sQ  )$ it is possibile to define an adjoint operator 
	\[ \nabla =[D, -]\colon  \Tot(\mathcal{U}, \Sad  ) \to \Tot(\mathcal{U},   \Omega_X^1[-1]\otimes\Sad ).\] 
	
 A cyclic form on the Atiyah Lie algebroid $\mathcal{Q}$ is a symmetric bilinear form $\la -,- \ra \colon \Sad \times \Sad \to \Oh_X$ such that for all $x,y \in \Sad$ and $q \in\mathcal{Q}$,
 \[ \la [q,x], y \ra + \la x, [q, y] \ra = \rho (q) (\la x, y \ra),\]
 where $\rho \colon \mathcal{Q} \to \Theta_X$ is the anchor map of the Atiyah Lie algebroid $\mathcal{Q}$.
 
 \begin{example}
 	The cyclic form induced by the adjoint representation of a DG-Lie algebroid of Example~\ref{ex.ad fasci} in this case can be constructed in an equivalent way, starting from the Killing form of the Lie algebra $\g$
 of the group $G$: $$K \colon \g \otimes_\K \g \to \K, \quad K (g,h) = \Tr (\ad g \ad h).$$
 Take $x,y$ in $\Sad(U)$ and let $U= \bigcup_{i} U_i$ with $U_i$ open sets trivialising the principal bundle $P$,
 then $$x = \{x_i \colon U_i \to \g \ | \ x_i (p)= \Ad_{g_{ij}(p)} x_j(p) \ \ \forall p \in U_{ij}\},$$ and analogously for $y$. Define $\la x, y \ra $ as $\{\la x_i, y_i \ra \colon U_i \to \K \}$, where for $p \in U_i$, $$\la x_i, y_i \ra (p) = K (x_i(p), y_i (p)) .$$ This is well defined because the Killing form is invariant under automorphisms of the Lie algebra $\g$, so that for $p \in U_{ij}$
 \[ K (x_i(p), y_i(p)) = K(\Ad_{g_{ij}(p)} x_j(p), \Ad_{g_{ij}(p)} y_j(p))= K (x_j(p), y_j(p)).\]
 \end{example}

	\medspace
	
	Recall that $\Tot$ preserves multiplicative structures, therefore $\Tot(\mathcal{U}, \Sad  )$ is a DG-Lie algebra.  In the sequel, $\Tot(\sU, \Omega_X^{\leq 1}[2])= \Tot (\mathcal{U}, \Oh_X[2] \xrightarrow{\dr} \Omega_X^1[1] )$ is considered as a DG-Lie algebra with trivial bracket; its differential is denoted  $d_{\Tot } + \dr$.
	 Theorem~\ref{teo.linfinitoalg} then yields the following.
	\begin{corollary}\label{cor.linf per adP}
		For every simplicial {connection} $D$ on a  principal bundle $P$ on a  smooth separated scheme $X$ of finite type over an algebraically closed field $\K$ of characteristic zero, endowed with a $d_{\Tot}$-closed cyclic form $\la-,-\ra \colon \Tot(\mathcal{U}, \Omega^i_X[-i] \otimes \Sad) \times \Tot(\mathcal{U}, \Omega^j_X[-j] \otimes \Sad) \to \Tot(\mathcal{U}, \Omega^{i+j}_X[-i-j] )$, $i,j \geq 0$, there exists an $L_\infty$ morphism of DG-Lie algebras on the field $\K$
		\[ f \colon \Tot (\mathcal{U}, \Sad) \rightsquigarrow \Tot(\sU, \Omega_X^{\leq 1}[2]),\]
		with components 
		\begin{align*}
		f_1 (x) &= \langle u, x \rangle,\\
		f_2 (x,y) &= \frac{1}{2} \big( \la \nabla(x), y \ra - (-1)^{\overline{x}\ \overline{y}} \la \nabla (y), x \ra \big),\\
		f_3(x,y,z)  &=-\frac{1}{2} \la x, [y,z] \ra,\\
		f_n &= 0 \ \forall n \geq 4.
		\end{align*}
	\end{corollary}
	
	As seen in Remark~\ref{rem.primacomp}, the linear component $f_1$ of the $L_\infty$ morphism induces a map of graded Lie algebras 
	\[ f_1 \colon H^*(\Tot(\mathcal{U}, \Sad)) \to H^* (\Tot(\sU, \Omega_X^{\leq 1}[2])),\]
	which, since the open cover $\mathcal{U}$ is affine, becomes
	\[ f_1 \colon H^* (X, \Sad) \to \mathbb{H}^* (X, \Omega_X^{\leq 1}[2]).\]

	\begin{corollary}\label{cor.principali}
	Let $P$ be a principal bundle on a smooth separated scheme $X$ of finite type over an algebraically closed field $\K$ of characteristic zero and let $$\la-,-\ra \colon \Tot(\mathcal{U}, \Omega^i_X[-i] \otimes \Sad) \times \Tot(\mathcal{U}, \Omega^j_X[-j] \otimes \Sad) \to \Tot(\mathcal{U}, \Omega^{i+j}_X[-i-j] ),\quad i,j \geq 0,$$ be a $d_{\Tot}$-closed cyclic form.
	Then every obstruction to the deformations of $P$ belongs to the kernel of the map
	\begin{equation*}
	f_1 \colon H^2 (X, \Sad) \to \mathbb{H}^2 (X, \Omega_X^{\leq 1}[2]), \quad f_1(x)= \la \At(P), x \ra,
	\end{equation*}
	where $\At(P)$ denotes the Atiyah class of the principal bundle $P$.
	\end{corollary}
\begin{proof}
	The proof is analogous to the one of Corollary~\ref{cor.fasci2}: the linear component of the $L_\infty$ morphism of DG-Lie algebras of Corollary~\ref{cor.linf per adP}  induces a morphism in cohomology which commutes with obstruction maps of the associated deformation functors, and the deformation functor associated to an abelian DG-Lie algebra has trivial obstructions. 
	By Corollary~\ref{cor.dgladef}, 	
	if $\mathcal{U}= \{ U_i\}$ is an affine open cover of $X$,  the DG-Lie algebra $\Tot(\mathcal{U}, \Sad)$ controls the deformations of $P$ and
	 an obstruction space is $H^2(\Tot(\mathcal{U}, \Sad)) \cong H^2 (X, \Sad)$. Since the DG-Lie algebra $\Tot(\sU, \Omega_X^{\leq 1}[2])$ is abelian, we obtain that $f_1$ annihilates all obstructions.
	
\end{proof}	 
	 
	 \begin{ackno} I thank my supervisor Marco Manetti for the help during the preparation of this paper. 	 \end{ackno}


\begin{thebibliography}{99}
		\bibitem{Arta} Artamkin, I.V.: \textit{On deformations of sheaves}. Math. USSR Izvestiya \textbf{32} No. 3,  663-668  (1989).
		
		\bibitem{At}
		{{Atiyah}, M.~F.:}
		{\em Complex analytic {connection}s in fibre bundles.}
		Trans. Amer. Math. Soc. {\bf 85} (1957), 181--207.
		
		\bibitem{pairs}
		Batakidis, P., Voglaire, Y.: {\em Atiyah classes and dg-Lie algebroids for matched pairs}, Journal of Geometry and Physics, 2018, Volume 123, 156-172,
		\href{https://arxiv.org/abs/1601.06254}{ 	arXiv:1601.06254}.
		
		\bibitem{Bis}
		{Biswas, I., Ramanan, S.:} \textit{An Infinitesimal Study of the Moduli of Hitchin Pairs.} Journal of the London Mathematical Society, Volume 49, Issue 2, April 1994, Pages 219–231
		
		\bibitem{Blo} 
		Bloch, S.: \textit{Semi-regularity and de Rham cohomology}. Invent.
		Math. \textbf{17},  51-66 (1972).
		
		\bibitem{ClosedForms}
		Bottacin, F.:
	\textit{	Atiyah Classes and Closed Forms on Moduli Spaces of Sheaves},
		Rend. Sem. Mat. Univ. Padova, 121 (2009), 165-177.
		
		\bibitem{liealgebroids}
		Bottacin, F.: \textit{Atiyah classes of Lie algebroids}, ``Current Trends in Analysis and its Applications", Proceedings of the 9th ISAAC Congress, Krakow, 2013. Trends in Mathematics, Birkhauser (2015), 375-393.
		
		\bibitem{BF}  Buchweitz, R.-O.,   Flenner, H.: \textit{A Semiregularity Map for Modules and Applications to Deformations.} Compositio Mathematica \textbf{137}, 135-210, (2003) \href{https://arxiv.org/abs/math/9912245}{arXiv:9912245}.
		
		
		\bibitem{Don}
		{Donin, I. F.:} \textit{Construction of a versal family of deformations for holomorphic bundles over a compact complex space}, 1974 Math. USSR Sb. 23 405
		
		
		\bibitem{Dup} Dupont, J.L.: \textit{Curvature and characteristic classes}. Lecture Notes in Mathematics 640, Springer-Verlag (1978).
		
		\bibitem{Rational} Félix, Y., Halperin, S., Thomas, J.: \textit{Rational homotopy theory.} Graduate texts in mathematics 205,
		Springer-Verlag (2001).
		
		\bibitem{sheaves}
		Fiorenza, D., Iacono, D., Martinengo, E.: \textit{Differential graded Lie algebras controlling infinitesimal deformations of coherent sheaves}, J. Eur. Math. Soc. (JEMS) 14 (2012), no. 2, 521-540; \href{https://arxiv.org/abs/0904.1301}{arXiv:0904.1301}.
		
		\bibitem{semicos}
		Fiorenza, D., Manetti, M., Martinengo, E.: \textit{Semicosimplicial DGLAs in deformation theory}. Communications in Algebra Vol. 40, Iss. 6, 2012;
		\href{https://arxiv.org/abs/0803.0399}{arXiv:0803.0399}. 
		
		
		
		\bibitem{getzler04} {Getzler, E.:}
		\textit{Lie theory for nilpotent $L_{\infty}$ algebras.}
		Ann. of Math. \textbf{170} (1),  271-301 (2009),
		\href{https://arxiv.org/abs/math/0404003}{arXiv:math/0404003v4}.
		
		\bibitem{Hart}
	 	Hartshorne, R.: \textit{Algebraic geometry}, Graduate Texts in Mathematics, 52, Springer-Verlag, New York,	(1977).
		
		\bibitem{Hin}  Hinich, V.: \textit{Descent of Deligne groupoids}, Int. Math. Research Notices, 1997, n. 5, 223-239, \href{https://arxiv.org/abs/alg-geom/9606010}{alg-geom/9606010}
		
		
		
		
		\bibitem{semireg}
		Iacono, D.,  Manetti, M.: \textit{Semiregularity and obstructions of complete intersections.} Advances in Mathematics 235 (2013) 92-125. \href{https://arxiv.org/abs/1112.0425}{arXiv:1112.0425}. 
		
		
		\bibitem{DMcoppie}
		Iacono, D.,  Manetti, M.: \textit{On deformations of pairs (manifold, coherent sheaf).} Canad. J. Math. \textbf{71},
		1209-1241 (2019); \href{https://arxiv.org/abs/1707.06612}{arXiv:1707.06612}.
		
		\bibitem{K}  Kontsevich M.:
		\textit{Deformation quantization of Poisson manifolds, I}.
		Letters in Mathematical Physics~\textbf{66}, 157-216 (2003),
		\href{https://arxiv.org/abs/q-alg/9709040}{arXiv:q-alg/9709040v1}

		\bibitem{linfsemireg}
		Lepri, E., Manetti, M.: \textit{Connections and $L_\infty$ liftings of semiregularity maps} (2021), \href{https://arxiv.org/abs/2102.05016}{arXiv:2102.05016}
		
		\bibitem{luigi}
		Lunardon, L.: \textit{Some remarks on Dupont contraction}. Rend. Mat. Appl. 39, 79-96 (2018).
		\href{https://arxiv.org/abs/1807.02517}{arXiv:1807.02517}

		\bibitem{Mack}
		{Mackenzie,	K. C. H.:}
		\textit{General theory of Lie groupoids and Lie algebroids}, vol. 213 of London Mathematical Society Lecture Note Series, Cambridge University Press, Cambridge, 2005.
		
		\bibitem{ManRendiconti}  Manetti, M.:
		\textit{Lectures on deformations of complex manifolds}. Rend. Mat.
		Appl. (7) \textbf{24},  1-183 (2004); 
		\href{https://arxiv.org/abs/math/0507286}{ arXiv:math/0507286v1}
		
		\bibitem{Man}  Manetti, M.: \textit{Differential graded Lie algebras and formal deformation theory.} Algebraic Geometry: Seattle 2005. Proc. Sympos. Pure Math. 80, 785-810 (2009)

		
		\bibitem{LMDT} Manetti, M.:
	\textit{	Lie methods in deformation theory}. Forthcoming book, 
		Draft version 2021. 
		
		\bibitem{Meazz}  Meazzini, F.: \textit{A DG-enhancement of $D(\operatorname{QCoh}(X))$ with applications in deformation theory} (2018), 
		\href{https://arxiv.org/abs/1808.05119}{arXiv:1808.05119}.
		
		
		\bibitem{Mehta} Mehta, R.A., Stiénon, M., Xu, P.: \textit{The Atiyah class of a dg-vector bundle.} Comptes Rendus Mathematique 353 (2015), no. 4, 357-362, \href{https://arxiv.org/abs/1502.03119}{arXiv:1502.03119}
		
	\bibitem{Nav}	Navarro Aznar, V.:\textit{ Sur la théorie de Hodge-Deligne}. Invent. Math. 90, 11-76 (1987).
		
			\bibitem{Pri} Pridham, J.~P.: \textit{Semiregularity as a consequence of Goodwillie's theorem.} (2012);
		\href{https://arxiv.org/abs/1208.3111}{arXiv:1208.3111}



		
		\bibitem{Sim} 
		Simpson, C. T.: \textit{Moduli of representations of the fundamental group of
			a smooth projective variety II}, Publications mathématiques de l’I.H.É.S., tome 80 (1994), p. 5-79
		
		
		\bibitem{Sor}
		Sorger, C.: \textit{Lectures on moduli of principal $G$-bundles over algebraic curves}, (1999) available at
		\href{https://inis.iaea.org/collection/NCLCollectionStore/_Public/38/005/38005695.pdf}{https://inis.iaea.org/collection/NCLCollectionStore/\_Public/38/005/38005695.pdf} 
		
		\bibitem{Ueno}
		Ueno, K.: \textit{Infinitesimal deformation of principal bundles, determinant bundles and affine Lie algebras}
		 Notes on the Research Institute of Mathematical Analysis
		 Volume 778, 1992 p.36-41
		 
		  \bibitem{Whi} Whitney, H.: \textit{Geometric integration theory}. Princeton University Press, Princeton, N. J., (1957)
		
		
		
		
		
		
		
		
	\end{thebibliography}
\end{document}